\documentclass[12pt,a4paper]{amsart}
\usepackage{mathrsfs}

\usepackage{amsfonts}
\usepackage{amsmath}
\usepackage{amssymb}
\usepackage{enumerate}
\usepackage{hyperref}
\usepackage{latexsym}
\usepackage{xcolor}
\usepackage[shortlabels]{enumitem}
\setlist[enumerate,1]{label={\upshape(\roman*)}}

 \makeatletter
    
    \newcommand{\Rmnum}[1]
    {\expandafter\@slowromancap\romannumeral #1@}
    \makeatother

\newtheorem{thm}{Theorem}[section]

\newtheorem{prop}[thm]{Proposition}
\newtheorem{lemma}[thm]{Lemma}

\newtheorem{cor}[thm]{Corollary}

\newtheorem{example}[thm]{Example}
\newtheorem{defin}[thm]{Definition}

\theoremstyle{definition}
\newtheorem{remark}[thm]{Remark}

\DeclareMathOperator{\rad}{rad}

\title[Weakly distance-regular circulants, \Rmnum{1}]{Weakly distance-regular circulants, \Rmnum{1}}

\author[Munemasa]{Akihiro Munemasa}
\address{Research Center for Pure and Applied Mathematics\\Graduate School of Information Sciences\\Tohoku University\\Sendai 980--8579\\Japan}
\email{munemasa@math.is.tohoku.ac.jp}

\author[Wang]{Kaishun Wang}
\address{Laboratory of Mathematics and Complex Systems (MOE),~School of Mathematical Sciences\\Beijing Normal University\\Beijing 100875\\China}
\email{wangks@bnu.edu.cn}

\author[Yang]{Yuefeng Yang}
\address{School of Science\\China University of Geosciences\\Beijing 100083\\China}
\email{yangyf@cugb.edu.cn}

\author[Zhu]{Wenying Zhu}
\address{Sch. Math. Sci. {\rm \&} Lab. Math. Com. Sys.\\Beijing Normal University\\Beijing 100875\\China}
\email{zfwenying@mail.bnu.edu.cn}

\begin{document}

\begin{abstract}
We classify certain non-symmetric commutative association schemes. As an application, we determine all the weakly distance-regular circulants of one type of arcs by using Schur rings. We also give the classification of primitive weakly distance-regular circulants.
\end{abstract}

\keywords{association scheme; Schur ring; weakly distance-regular digraph; circulant;  primitivity}

\subjclass[2010]{05E30}

\maketitle
\section{Introduction}
As a natural directed version of distance-regular graphs (see \cite{AEB98,DKT16} for a background of the theory of distance-regular graphs), Wang and Suzuki \cite{KSW03} introduced the concept of weakly distance-regular digraphs. Both distance-regular graphs and weakly distance-regular digraphs can be regarded as association schemes.

A \emph{$d$-class association scheme} $\mathfrak{X}$ is a pair $(X,\{R_{i}\}_{i=0}^{d})$, where $X$ is a finite set, and each $R_{i}$ is a
nonempty subset of $X\times X$ satisfying the following axioms (see \cite{EB84,PHZ96,PHZ05} for a background of the theory of association schemes):
\begin{enumerate}
\item\label{as-1} $R_{0}=\{(x,x)\mid x\in X\}$ is the diagonal relation;

\item\label{as-2} $X\times X=R_{0}\cup R_{1}\cup\cdots\cup R_{d}$, $R_{i}\cap R_{j}=\emptyset~(i\neq j)$;

\item\label{as-3} for each $i$, $R_{i}^{\rm T}=R_{i^{*}}$ for some $0\leq i^{*}\leq d$, where $R_{i}^{\rm T}=\{(y,x)\mid(x,y)\in R_{i}\}$;

\item\label{as-4} for all $i,j,l$, the cardinality of the set $$P_{i,j}(x,y):=R_{i}(x)\cap R_{j^{*}}(y)$$ is constant whenever $(x,y)\in R_{l}$, where $R(x)=\{y\mid (x,y)\in R\}$ for $R\subseteq X\times X$ and $x\in X$. This constant is denoted by $p_{i,j}^{l}$.
\end{enumerate}
A $d$-class association scheme is also called an association scheme with $d$ classes (or even simply a scheme). The integers $p_{i,j}^{l}$ are called the \emph{intersection numbers} of $\mathfrak{X}$. We say that $\mathfrak{X}$ is \emph{commutative} if $p_{i,j}^{l}=p_{j,i}^{l}$ for all $i,j,l$. The subsets $R_{i}$ are called the \emph{relations} of $\mathfrak{X}$. For each $i$, the integer $k_{i}:=p_{i,i^{*}}^{0}$ is called the \emph{valency} of $R_{i}$.   A relation $R_{i}$ is called \emph{symmetric} if $i=i^{*}$, and \emph{non-symmetric} otherwise. An association scheme is called \emph{symmetric} if all relations are symmetric, and \emph{non-symmetric} otherwise. An association
scheme is called \emph{skew-symmetric} if the diagonal relation is the
only symmetric relation.

Let $\mathfrak{X}=(X,\{R_{i}\}_{i=0}^{d})$ be an association scheme. For two nonempty subsets $E$ and $F$ of $\{R_{i}\}_{i=0}^{d}$, define
\[EF:=\{R_{l}\mid\sum_{R_{i}\in E}\sum_{R_{j}\in F}p_{i,j}^{l}\neq0\}.\]
We write $R_{i}R_{j}$ instead of $\{R_{i}\}\{R_{j}\}$,
and $R_i^2$ instead of $\{R_{i}\}\{R_{i}\}$.
If $R_{i^{*}}R_{j}\subseteq F$ for any $R_{i},R_{j}\in F$, we say that $F$ is {\em closed}. We say that $R_{j}$ {\em generates} $\mathfrak{X}$ if the smallest
closed subset containing $R_{j}$ is equal to $\{R_{i}\}_{i=0}^{d}$. We call $\mathfrak{X}$ \emph{primitive} if every non-diagonal relation generates $\mathfrak{X}$. Otherwise, $\mathfrak{X}$ is said to be \emph{imprimitive}.

In this paper, we classify certain non-symmetric commutative association schemes. Our first main result is as follows. See Section 2 for precise definitions of a pseudocyclic scheme and a cyclotomic scheme.

\begin{thm}\label{main1}
Let $\mathfrak{X}=(X,\{R_{i}\}_{i=0}^{d})$ be a commutative  association scheme generated by a non-symmetric relation $R_{1}$ with $k_1>1$ satisfying
\begin{align}
R_{1}^{2}&\subseteq\{R_{1},R_{1^{*}},R_{2}\},\label{1.1a}\\
R_{1}R_{1^{*}}&\subseteq\{R_{0},R_{1},R_{1^{*}},R_{2},R_{2^{*}}\},\label{1.1b}\\
2&\notin\{1^{*},2^{*}\} .\label{1.1c}
\end{align}
If $k_{1}\geq k_2$, then $d=4$. Moreover, if $\mathfrak{X}$ is pseudocyclic, then $\mathfrak{X}$ is isomorphic to the cyclotomic scheme ${\rm Cyc}(13,4)$.
\end{thm}


In \cite{SM03}, Miklavi\'c and Poto\v cnik gave the classification of distance-regular circulants. As an application of Theorem \ref{main1}, we obtain the classification of weakly distance-regular circulants of one type of arcs. In order to state our results, we introduce some basic notations and terminologies about weakly distance-regular circulants. See \cite{YF22,HS04,KSW03,KSW04,YYF16,YYF18,YYF20,YYF22,QZ} for more details.

We always use $\Gamma$ to denote a finite simple digraph, which is not undirected. We write $V\Gamma$ and $A\Gamma$ for the vertex set and arc set of
$\Gamma$, respectively. For any $x,y\in V\Gamma$, let $\partial_{\Gamma}(x,y)$ be the \emph{one-way distance} from $x$ to $y$ in $\Gamma$, and ${\tilde{\partial}}_{\Gamma}(x,y):=(\partial_{\Gamma}(x,y),\partial_{\Gamma}(y,x))$ the \emph{two-way distance} from $x$ to $y$ in $\Gamma$. If no confusion occurs, we write $\partial(x,y)$ (resp. $\tilde{\partial}(x,y)$) instead of $\partial_\Gamma(x,y)$ (resp. ${\tilde{\partial}}_\Gamma(x,y)$). Denote $\tilde{\partial}(\Gamma)=\{\tilde{\partial}(x,y)\mid x,y\in V\Gamma\}$ and $\Gamma_{\tilde{i}}=\{(x,y)\in V\Gamma\times V\Gamma\mid\tilde{\partial}(x,y)=\tilde{i}\}$. An arc $(u,v)$ of $\Gamma$ is of \emph{type} $(1,r)$ if $\partial(v,u)=r$.

A strongly
connected digraph $\Gamma$ is said to be \emph{weakly distance-regular} if the configuration $\mathfrak{X}(\Gamma)=(V\Gamma,\{\Gamma_{\tilde{i}}\}_{\tilde{i}\in\tilde{\partial}(\Gamma)})$ is an association scheme. This is a directed generalization of distance-regular graphs (see \cite{AEB98}). We call $\mathfrak{X}(\Gamma)$ the \emph{attached scheme} of $\Gamma$. We say that $\Gamma$ is \emph{primitive} (resp. \emph{commutative}) if $\mathfrak{X}(\Gamma)$ is primitive (resp. commutative). For each $(a,b)\in\tilde{\partial}(\Gamma)$, we write $\Gamma_{a,b}$ (resp. $k_{a,b}$) instead of $\Gamma_{(a,b)}$ (resp. $k_{(a,b)}$).

Let $G$ be a finite, multiplicatively written, group with identity $e$, and $S$ be a subset of $G$ without identity. A \emph{Cayley digraph} of a group $G$ with respect
to the set $S$, denoted by $\textrm{Cay}(G,S)$, is the digraph with vertex set $G$, where
$(x,y)$ is an arc whenever $yx^{-1}\in S$. A digraph which is isomorphic to a Cayley digraph of a cyclic group is called a \emph{circulant}. For a positive integer $n$, denote by $C_{n}$ the Cayley digraph ${\rm Cay}(\mathbb{Z}_{n},\{1\})$, $K_{n}$ the complete graph ${\rm Cay}(\mathbb{Z}_{n},\mathbb{Z}_{n}\setminus\{0\})$
and $\overline{K}_{n}$ the edgeless graph ${\rm Cay}(\mathbb{Z}_{n},\emptyset)$, respectively.

A classical example of weakly distance-regular Cayley digraphs are Paley digraphs, defined as follows. For a prime power
$q=p^{m}$ with $q\equiv 3~(\textrm{mod}~4)$, let $\textrm{GF}(q)$ denote the finite field of cardinality $q$, let $S$ denote the set of all squares in the multiplicative group of $\textrm{GF}(q)$, and let $G\cong\mathbb{Z}_{p}^{m}$ denote the additive group
of $\textrm{GF}(q)$, where $\mathbb{Z}_{p}=\{0,1,\ldots,p-1\}$ is the cyclic group of order $p$, written additively. The \emph{Paley digraph $P(q)$ of order $q$} is defined as the digraph ${\rm Cay}(G,S)$.

For the digraphs $\Gamma$ and $\Sigma$, the {\em direct product} $\Gamma\times\Sigma$ is the  digraph with vertex set $V\Gamma\times V\Sigma$ such that $((u_1,u_2),(v_1,v_2))$ is an arc if and only if $(u_1,v_1)\in A\Gamma$ and $(u_2,v_2)\in A\Sigma$; the {\em lexicographic product} $\Gamma[\Sigma]$ is the digraph with vertex set $V\Gamma\times V\Sigma$ such that $((u_1,u_2),(v_1,v_2))$ is an arc if and only if either $(u_1,v_1)\in A\Gamma$, or
$u_1=v_1$ and $(u_2,v_2)\in A\Sigma$.  

Our second main theorem is as follows, which classifies all weakly distance-regular circulants of one type of arcs.

\begin{thm}\label{Main2}
A digraph
$\Gamma$ is a weakly distance-regular circulant of one type of arcs,
if and only if
$\Gamma$ is isomorphic to one of the following digraphs:
\begin{enumerate}
\item\label{Main2-1} $C_m[\overline{K}_{l}]$;

\item\label{Main2-2}$P(p)[\overline{K}_{l}]$;

\item\label{Main2-3} $C_3\times K_{h}$;

\item\label{Main2-4} ${\rm Cay}(\mathbb{Z}_{13},\{1,3,9\})[\overline{K}_{l}]$.
\end{enumerate}
Here, $l\geq1$, $m\geq3$, $h,p>3$, $3\nmid h$ and $p$ is prime such that $p\equiv 3\pmod4$.
\end{thm}

In the forthcoming paper \cite{AM}, we shall classify weakly distance-regular circulants of more than one types of arcs. The following result is a consequence of Theorem \ref{Main2}, which determines all primitive weakly distance-regular circulants.

\begin{cor}\label{Main3}
A digraph
$\Gamma$ is a primitive weakly distance-regular circulant,
if and only if
$\Gamma$ is isomorphic to $C_{p}$, $P(r)$ or ${\rm Cay}(\mathbb{Z}_{13},\{1,3,9\})$, where $p,r$ are primes with $r\equiv 3\pmod4$.
\end{cor}

The remainder of this paper is organized as follows.
In Section 2, we give some basic results concerning association schemes and weakly distance-regular digraphs. In Section 3, we give a proof of Theorem \ref{main1} based on the results in Section 2. In Section 4, we recall the concepts of Schur rings. In Sections 5 and 6, we discuss the subdigraphs of weakly distance-regular circulants by using Theorem \ref{main1} and the results in Section 4. In Section 7, we prove Theorem \ref{Main2} and Corollary \ref{Main3} based on the results in Sections 5 and 6.

\section{Preliminaries}

Let $\mathfrak{X}=(X,\{R_{i}\}_{i=0}^{d})$ be a commutative association scheme with $|X|=n$. The \emph{adjacency matrix} $A_{i}$ of $R_{i}$ is the $n\times n$ matrix whose $(x,y)$-entry is $1$ if $(x,y)\in R_{i}$, and $0$ otherwise. By the \emph{adjacency} or \emph{Bose-Mesner algebra} $\mathfrak{U}$ of $\mathfrak{X}$ we mean the algebra generated by $A_{0},A_{1},\ldots,A_{d}$ over the complex field. Axioms \ref{as-1}--\ref{as-4} are equivalent to the following:
\[A_{0}=I,\quad \sum_{i=0}^{d}A_{i}=J,\quad A_{i}^{\rm T}=A_{i^{*}},
\quad A_{i}A_{j}=\sum_{l=0}^{d}p_{i,j}^{l}A_{l},\]
where $I$ and $J$ are the identity and all-one matrices of order $n$, respectively.

Since $\mathfrak{U}$ consists of commuting normal matrices, it has a second basis consisting of primitive idempotents $E_{0}=J/n, E_{1},\ldots,E_{d}$. The integers $m_{i}=\textrm{rank}E_{i}$ are called the \emph{multiplicities} of $\mathfrak{X}$, and $m_{0}=1$ is said to be the trivial multiplicity. A commutative association scheme is called \emph{pseudocyclic} if all the non-trivial multiplicities coincide with each other.

Now we list basic properties of intersection
numbers.

\begin{lemma}[{\cite[Chapter II, Proposition 2.2]{EB84}} and {\cite[Proposition 5.1]{ZA99}}]\label{jb} Let the configuration $(X,\{R_{i}\}_{i=0}^{d})$ be an association scheme. The following hold:
\begin{enumerate}
\item\label{jb-1} $k_{i}k_{j}=\sum_{l=0}^{d}p_{i,j}^{l}k_{l}$;

\item\label{jb-2} $p_{i,j}^{l}k_{l}=p_{l,j^{*}}^{i}k_{i}=p_{i^{*},l}^{j}k_{j}$;

\item\label{jb-3}$\sum_{j=0}^{d}p_{i,j}^{l}=k_{i}$;

\item\label{jb-4} $\sum_{\alpha=0}^{d}p_{i,j}^{\alpha}p_{f,\alpha}^{l}=\sum_{\beta=0}^{d}p_{f,i}^{\beta}p_{\beta,j}^{l}$;

\item\label{jb-5} ${\rm lcm}(k_{i},k_{j})\mid p_{i,j}^{l}k_{l}$.
\end{enumerate}
\end{lemma}

We call an association scheme $(X,\{R_{i}\}_{i=0}^{d})$ a \emph{translation scheme}, if the underlying set $X$ has the structure of an additive group,  and for all relations $R_{i}$,
\[(x,y)\in R_{i}\Longleftrightarrow(x+z,y+z)\in R_{i}~\textrm{for all~$z\in X$}.\]

A classical example of translation schemes is the cyclotomic scheme which we describe now. Let $\alpha$ be a primitive element of $\textrm{GF}(q)$, where $q$ is a prime power. For a fixed divisor $d$ of $q-1$, define
\[(x,y)\in R_{i}~~\textrm{if}~~x-y\in\alpha^{i}\langle\alpha^{d}\rangle,~1\leq i\leq d.\]
These relations $R_{i}$ define a pseudocyclic scheme, denoted by $\textrm{Cyc}(q,d)$, called the $d$-class \emph{cyclotomic scheme} over ${\rm GF}(q)$.

\begin{lemma}[{\cite[Theorem 2.10.5]{AEB98}}]\label{prime}
Let $(X,\{R_{i}\}_{i=0}^{d})$ be a primitive translation scheme with $d\geq2$. If $X$ has some cyclic Sylow subgroup (and in particular if $X$ is cyclic), then
$|X|$ is prime and $(X,\{R_{i}\}_{i=0}^{d})$ is cyclotomic.
\end{lemma}

\begin{lemma}\label{lem:1}
Let $({\rm GF}(q),\{R_{i}\}_{i=0}^{d})$ be the $d$-class cyclotomic scheme. For $1\leq i\leq d$ and $s\in{\rm GF}(q)\setminus\{0\}$, define
\begin{align}
R_{i}^{(s)}=\{(sx,sy)\mid(x,y)\in R_{i}\}.\nonumber
\end{align}
Then, for any fixed $i\in\{1,2,\ldots,d\}$,
\begin{align}
\{R_{i}^{(s)}\mid s\in{\rm GF}(q)\setminus\{0\}\}=\{R_{j}\mid1\leq j\leq d\}.\nonumber
\end{align}
\end{lemma}
\begin{proof}
Let $\alpha$ be a primitive element of ${\rm GF}(q)$. Then
\begin{align}
\{s\alpha^{i}\langle\alpha^{d}\rangle\mid s\in{\rm GF}(q)\setminus\{0\}\}=\{\alpha^{j}\langle\alpha^{d}\rangle\mid1\leq j\leq d\},\nonumber
\end{align}
and the result follows immediately from this.
\end{proof}

As a consequence of Lemma~\ref{lem:1}, as digraphs, we have $(\textrm{GF}(q),R_i)\cong (\textrm{GF}(q),R_j)$
for $1\leq i,j\leq d$. For example when $(q,d)=(13,4)$ we have $(\textrm{GF}(13),R_i)\cong \textrm{Cay}(\mathbb{Z}_{13},\{1,3,9\})$ for $1\leq i\leq 4$.

\begin{lemma}\label{skew}
Let $\mathfrak{X}=(X,\{R_{0},R_{1},R_{1^{*}},R_{2},R_{2^{*}}\})$ be a $4$-class skew-symmetric and pseudocyclic association scheme. If $|X|>5$ and $p_{1,1}^{2^{*}}=0$, then $\mathfrak{X}$ is isomorphic to ${\rm Cyc}(13,4)$.
\end{lemma}
\begin{proof}
The symmetrization $(X,\{R_{0},R_{1}\cup R_{1^{*}},R_{2}\cup R_{2^{*}}\})$ of $\mathfrak{X}$ is a pseudocyclic association scheme with 2 classes. Thus, $\mathfrak{X}$ is a skew-symmetric fission of a conference graph (see \cite{JM11} for definitions). By the proof of \cite[Theorem 3.3]{JM11}, there exist integers $u$ and $v$ such that
\begin{align}
|X|&=u^{2}+4v^{2},\label{skew1}\\
p_{1,1}^{2^{*}}&=\frac{|X|+1+2u\pm8v}{16}\nonumber.
\end{align}
Since $p_{1,1}^{2^{*}}=0$, we have $v^{2}=(|X|+1+2u)^{2}/64$. Substituting this into \eqref{skew1}, we obtain
\[20u^{2}+4(|X|+1)u+|X|^{2}-14|X|+1=0.\]
Since $u$ is an integer, the discriminant of the quadratic is not less than zero. Thus
\[(|X|+1)^{2}-5(|X|^{2}-14|X|+1)\geq0.\]
It follows that $9-4\sqrt{5}\leq |X|\leq9+4\sqrt{5}$. By \cite[Theorem 3.3]{JM11}, one gets $|X|\equiv 5~(\textrm{mod}~8)$, which implies $|X|=13$. In view of \cite[Result 1]{MH96}, $\mathfrak{X}$ is isomorphic to $\textrm{Cyc}(13,4)$.
\end{proof}

We close this section with three results on weakly distance-regular digraphs.

\begin{lemma}\label{commutativity}
Let $\Gamma$ be a commutative weakly distance-regular digraph. Suppose that $(x_0,x_1,\ldots,x_m)$ is a path of length $m>0$. For each $i\in\{0,1,\ldots,m-1\}$, there exists a path $(x_0=y_0,y_1,\ldots,y_m=x_m)$ such that $\tilde{\partial}(x_0,y_1)=\tilde{\partial}(x_i,x_{i+1})$.
\end{lemma}
\begin{proof} The desired result follows from the commutativity of $\Gamma$.
\end{proof}

\begin{lemma}\label{lem}
Let $\Gamma$ be a weakly distance-regular digraph. If $p_{(1,q-1),(1,q-1)}^{(2,q-2)}>0$ with $q\geq3$, then each arc of type $(1,q-1)$ is contained in a circuit of length $q$ consisting of arcs of type $(1,q-1)$.
\end{lemma}
\begin{proof} Let $(x_{0},x_{1})\in\Gamma_{1,q-1}$. Since $p_{(2,q-2),(q-1,1)}^{(1,q-1)}k_{1,q-1}=p_{(1,q-1),(1,q-1)}^{(2,q-2)}k_{2,q-2}>0$ from Lemma \ref{jb} \ref{jb-2}, there exists a vertex $x_{2}\in P_{(2,q-2),(q-1,1)}(x_{0},x_1)$. Pick a path $(x_2,x_3,\ldots,x_{q-1},x_0)$. The fact that $(x_{1},x_{3})\in\Gamma_{2,q-2}$ implies that there exists a vertex $x_{2}'\in P_{(1,q-1),(1,q-1)}(x_{1},x_{3})$. Then $(x_{0},x_{1},x_{2}',x_3,\ldots,x_{q-1})$ is a circuit containing at least three arcs of type $(1,q-1)$. Repeating this process, there exists a circuit $(x_{0},x_{1},x_{2}',x_3',\ldots,x_{q-1}')$ consisting of arcs of type $(1,q-1)$. Since $(x_{0},x_1)\in\Gamma_{1,q-1}$ was arbitrary, the desired result follows.
\end{proof}

\begin{lemma}\label{lem:2}
Let $\Gamma$ be a weakly distance-regular digraph, and $(x_{0},x_{1},\ldots,x_{m-1})$ be a shortest circuit consisting of arcs of type $(1,q-1)$ with $q\geq3$. Assume $k_{1,q-1}\geq k_{a,b}$, where $(a,b)=\tilde{\partial}(x_0,x_2)$. Denote $Y_{i}=P_{(1,q-1),(1,q-1)}(x_{i-1},x_{i+1})$, where the indices are read modulo $m$. If $p_{(1,q-1),(1,q-1)}^{(2,q-2)}>0$ or $|\Gamma_{1,q-1}^{2}|=1$, then $(x_{i-1},x_{i+1})\in\Gamma_{a,b}$, and
\[P_{(a,b),(q-1,1)}(x_{i-2},x_{i-1})=P_{(q-1,1),(a,b)}(x_{i+1},x_{i+2})=Y_{i}\]
for all $i$. Moreover, if $q>3$ and $\Gamma$ is a Cayley digraph of an additive group, then $Y_{i}-x_{i-1}=Y_{i+1}-x_{i}$ for all $i$.
\end{lemma}
\begin{proof}Note that $m\geq q$. If $m=q$, then $\partial(x_i,x_{i+j})=j$ for $1\leq j<q$, and so $(a,b)=(2,q-2)$. If $m>q$, from Lemma \ref{lem}, then $p_{(1,q-1),(1,q-1)}^{(2,q-2)}=0$, and so $\Gamma_{1,q-1}^{2}=\{\Gamma_{a,b}\}$. Thus, $(x_{i-1},x_{i+1})\in\Gamma_{a,b}$ for all $i$.

To complete the proof of this lemma, we only need to consider the case $i=1$ only. For the first statement, since the proof of $Y_1=P_{(a,b),(q-1,1)}(x_{m-1},x_{0})$ is similar, It suffices to show $Y_{1}=P_{(q-1,1),(a,b)}(x_{2},x_{3})$ only. If $y\in Y_{1}$, then
$(x_{0},y,x_{2},\ldots,x_{m-1})$ is also a shortest circuit consisting of arcs of type $(1,q-1)$. Similarly, $(y,x_3)\in\Gamma_{a,b}$, and so $y\in P_{(q-1,1),(a,b)}(x_{2},x_{3})$. Hence, $Y_{1}\subseteq P_{(q-1,1),(a,b)}(x_{2},x_{3})$, and so $p_{(1,q-1),(1,q-1)}^{(a,b)}\leq p_{(q-1,1),(a,b)}^{(1,q-1)}$. By Lemma \ref{jb} \ref{jb-2}, one gets $p_{(q-1,1),(a,b)}^{(1,q-1)}k_{1,q-1}=p_{(1,q-1),(1,q-1)}^{(a,b)}k_{a,b}$. Since $k_{1,q-1}\geq k_{a,b}$, we obtain $p_{(q-1,1),(a,b)}^{(1,q-1)}\leq p_{(1,q-1),(1,q-1)}^{(a,b)}$. It follows that $p_{(1,q-1),(1,q-1)}^{(a,b)}=p_{(q-1,1),(a,b)}^{(1,q-1)}$, and so $Y_1=P_{(q-1,1),(a,b)}(x_{2},x_{3})$. Thus, the first statement is valid.

Now suppose that $q>3$ and $\Gamma$ is a Cayley digraph of an additive group. From the first statement, we have
\begin{align}\label{sec4-1}
Y_{1}=P_{(a,b),(q-1,1)}(x_{m-1},x_{0}).
\end{align}
Let $y\in Y_{2}$. Since $m\geq q>3$, $(x_{1},y,x_{3},\ldots,x_{m-1},x_{0})$ is a shortest circuit consisting of arcs of type $(1,q-1)$ which contains two distinct vertices $y$ and $x_{m-1}$. From the first statement again, one gets
\begin{align}\label{sec4-2}
P_{(1,q-1),(1,q-1)}(x_{0},y)=P_{(a,b),(q-1,1)}(x_{m-1},x_{0}).
\end{align}
Since $\Gamma$ is a Cayley digraph and $(x_1,y),(x_0,x_1)\in\Gamma_{1,q-1}$, we obtain
\begin{align}
x_{0}+y-x_{1}&\in\Gamma_{1,q-1}(x_0)\cap\Gamma_{q-1,1}(y)\nonumber\\
&=P_{(1,q-1),(1,q-1)}(x_{0},y)\nonumber\\
&=Y_{1}\nonumber
\end{align}
by \eqref{sec4-1} and \eqref{sec4-2}. This proves $y-x_{1}\in Y_{1}-x_{0}$. Since $y\in Y_{2}$ was arbitrary, one has $Y_{2}-x_{1}\subseteq Y_{1}-x_{0}$. The fact that $|Y_{1}|=|Y_{2}|$ implies $Y_{1}-x_{0}=Y_{2}-x_{1}$. The second statement is also valid.
\end{proof}

\section{Proof of Theorem~\ref{main1}}

We note that the second statement of Theorem~\ref{main1} follows immediately from Lemma~\ref{skew}. To prove the first statement of Theorem~\ref{main1}, we set $I:=\{i\mid R_{i}\in R_{1}^{2}\}$ and $J:=\{i\mid R_{i}\in R_{1}R_{1^{*}}\}$. Then $I\subseteq\{1,1^{*},2\}$ by (\ref{1.1a}). Since $R_{1}R_{1^{*}}=R_{1^{*}}R_{1}$ from the commutativity of $\mathfrak{X}$, we get $J=J^{*}$. Since $k_1>1$, we have $\{0\}\subsetneqq J$. It follows from \eqref{1.1b} that
\begin{align}
J=\{0,1,1^{*}\},~\{0,2,2^{*}\}~\textrm{or}~\{0,1,1^{*},2,2^{*}\}.\label{set J}
\end{align}

In the proof of Theorem~\ref{main1}, we also need the following auxiliary lemmas.

\begin{lemma}\label{3.2}
The following conditions are equivalent:
\begin{itemize}
\item [{\rm(i)}] $1\in I$;

\item [{\rm(ii)}] $p_{1,1}^{1}\neq0$;

\item [{\rm(iii)}] $\{1,1^{*}\}\subseteq J$.
\end{itemize}
\end{lemma}
\begin{proof}
By setting $i=j=l=1$ in Lemma \ref{jb} \ref{jb-2}, we get $p_{1,1}^1=p_{1,1^*}^1=p_{1,1^*}^{1^*}$. The desired result follows.
\end{proof}

\begin{lemma}\label{R2}
Let $\{R_{i}\mid i\in K\}$ be a subset of $\{R_{i}\}_{i=0}^{d}$. Fix $(x,z)\in R_{2}$. Then
\begin{enumerate}
\item\label{R2-1} $R_{2}R_{1}\subseteq\{R_{i}\mid i\in K\}$ if and only if $(x,w)\in \bigcup_{i\in K}R_{i}$ for all $w\in R_{1}(z)$;

\item\label{R2-2} $R_{2^{*}}R_{1}\subseteq\{R_{i}\mid i\in K\}$ if and only if $(z,w)\in \bigcup_{i\in K}R_{i}$ for all $w\in R_{1}(x)$.
\end{enumerate}
\end{lemma}
\begin{proof}
Note that $p_{2,1}^{i}\neq0$ is equivalent to $p_{i,1^{*}}^{2}\neq0$ by Lemma~\ref{jb} \ref{jb-2}, which in turn is equivalent to $R_{i}(x)\cap R_{1}(z)\neq\emptyset$. Thus, $R_{2}R_{1}\subseteq\{ R_{i}\mid i\in K\}$ is equivalent to $R_{1}(z)\subseteq \bigcup_{i\in K}R_{i}(x)$. This proves \ref{R2-1}. The proof of \ref{R2-2} is similar, hence omitted.
\end{proof}

\begin{lemma}\label{d=4}
Suppose that
\begin{align}
R_{2}R_{1}\subseteq\{R_{1},R_{1^{*}},R_{2},R_{2^{*}}\},\label{eq-1}\\
R_{2^{*}}R_{1}\subseteq\{R_{1},R_{1^{*}},R_{2},R_{2^{*}}\}.\label{eq-2}
\end{align}
Then $d=4$.
\end{lemma}
\begin{proof}
By \eqref{1.1a} and \eqref{1.1b},
\begin{align}
R_{1}^{3}\subseteq\{R_{1},R_{1^{*}},R_{2}\}R_{1}\subseteq\{R_{0},R_{1},R_{1^{*}},R_{2},R_{2^{*}}\}.\nonumber
\end{align}
Then
\begin{align}
R_{1}^{4}\subseteq\{R_{0},R_{1},R_{1^{*}},R_{2},R_{2^{*}}\}R_{1}\subseteq\{R_{0},R_{1},R_{1^{*}},R_{2},R_{2^{*}}\}.\nonumber
\end{align}
It follows from induction that
\begin{align}
R_{1}^{i}&\subseteq\{R_{0},R_{1},R_{1^{*}},R_{2},R_{2^{*}}\}R_{1}~\textrm{for}~i\geq5.\nonumber
\end{align}
Since $R_{1}$ generates $\mathfrak{X}$, we obtain $d=4$.
\end{proof}

\begin{lemma}\label{1,1*}
If $I=\{1,1^*\}$, then $d=4$.
\end{lemma}
\begin{proof}
If $J=\{0,1,1^*\}$, then $R_{1}^i=\{R_{0},R_{1},R_{1^*}\}$ for $i\geq3$, which implies $2=1^*$ since $R_1$ generates $\mathfrak{X}$, contrary to \eqref{1.1c}. Lemma \ref{3.2} and \eqref{set J} imply that $J=\{0,1,1^*,2,2^*\}$.

Fix $(x,z)\in R_2$. Pick elements $w\in R_1(z)$ and $w'\in R_{1}(x)$. Since $p_{1^*,1}^2=p_{1,1^*}^2\neq0$, there exists an element $y\in P_{1^*,1}(x,z)$. The fact that $I=\{1,1^*\}$ implies $(y,w),(y,w')\in R_{1}\cup R_{1^*}$. Since $2\notin\{1,1^*\}$ and $J=\{0,1,1^*,2,2^*\}$, one has $(x,w),(w',z)\in R_{1}\cup R_{1^*}\cup R_{2}\cup R_{2^*}$. \eqref{eq-1} and \eqref{eq-2} are valid from Lemma \ref{R2} \ref{R2-1} and \ref{R2-2}, respectively. The desired result follows by Lemma \ref{d=4}.
\end{proof}

\begin{lemma}\label{k1>k2}
If $k_1>k_2$ and $2\in I$, then $d=4$.
\end{lemma}
\begin{proof}
By Lemma \ref{d=4}, it suffices to show that \eqref{eq-1} and \eqref{eq-2} are both valid.

Indeed, fix $(x,z)\in R_2$. Pick an element $w\in R_1(z)$. By Lemma \ref{jb} \ref{jb-2}, we have $p_{1,1}^2k_2=p_{1^*,2}^1k_1$. Since $2\in I$ and $k_1>k_2$, one gets $p_{1,1}^2>p_{1^*,2}^1$. It follows that there exists an element $y\in P_{1,1}(x,z)$ such that $(y,w)\notin R_2$. Since $I\subseteq\{1,1^*,2\}$, we obtain $(y,w)\in R_{1}\cup R_{1^*}$. Note that $2\neq 1^*$ by \eqref{1.1c}. In view of \eqref{set J}, one has $(x,w)\in R_1\cup R_{1^*}\cup R_2\cup R_{2^*}$. \eqref{eq-1} follows from Lemma \ref{R2} \ref{R2-1}.

Next, pick an element $w'\in R_1(x)$. If $J=\{0,1,1^*\}$, then $(y,w')\in R_{1}\cup R_{1^*}$, and $(z,w')\in R_1\cup R_{1^*}\cup R_2\cup R_{2^*}$ since $I\subseteq\{1,1^*,2\}$, which imply that \eqref{eq-2} is valid from Lemma \ref{R2} \ref{R2-2}. Suppose $2,2^*\in J$. Lemma \ref{jb} \ref{jb-2} implies $p_{1,1^*}^2k_2=p_{1,2^*}^1k_1$. Since $k_1>k_2$, one gets $p_{1,1^*}^2>p_{1,2^*}^1$. The fact $J\subseteq\{0,1,1^*,2,2^*\}$ implies that there exists $y'\in P_{1,1^*}(x,z)$ with $(y',w')\in R_1\cup R_{1^*}\cup R_2$. Since $I\subseteq\{1,1^*,2\}$, from \eqref{eq-1}, one has $(z,w')\in R_{1}\cup R_{1^*}\cup R_2\cup R_{2^*}$. \eqref{eq-2} follows from Lemma \ref{R2} \ref{R2-2}.
\end{proof}

We list some consequences of Lemma~\ref{jb} \ref{jb-2} as follows:
\begin{align}
&p_{1,1}^{1}=p_{1,1^{*}}^{1}=p_{1^{*},1}^{1}=p_{1,1^{*}}^{1^*}=p_{1^{*},1}^{1^*},\label{2.1a}\\
&p_{1,1^{*}}^{2}=p_{1^{*},1}^{2}=p_{1,1^{*}}^{2^*}=p_{1^{*},1}^{2^*},\label{2.1b}\\
&p_{2,1}^{1}=p_{2^{*},1}^{1}.\label{2.1c}
\end{align}

\begin{lemma}\label{jiben}
The following hold:
\begin{enumerate}
\item\label{jiben-1} $\sum_{j=0}^{d}p_{1,j}^{l}=k_1$;

\item\label{jiben-2} $\sum_{l\in I}p_{1,1}^{l}k_l=k_1^2$;

\item\label{jiben-3} $2\sum_{l\in\{1,2\}}p_{1,1^{*}}^{l}k_l=k_1(k_1-1)$;

\item\label{jiben-4} $\sum_{\alpha\in I}(p_{1,1}^{\alpha})^{2}k_{\alpha}=k_1^2+2\sum_{\beta\in\{1,2\}}(p_{1,1^{*}}^{\beta})^{2}k_{\beta}$.
\end{enumerate}
\end{lemma}
\begin{proof} \ref{jiben-1} follows by setting $i=1$ in Lemma~\ref{jb} \ref{jb-3}. \ref{jiben-2} follows by setting $i=j=1$ in Lemma~\ref{jb} \ref{jb-1}. \ref{jiben-3} follows by setting $i=1$ and $j=1^{*}$ in Lemma~\ref{jb} \ref{jb-1}, and using \eqref{2.1a} and \eqref{2.1b}.

In view of Lemma \ref{jb} \ref{jb-2}, we have $p_{1^*,\alpha}^{1}=p_{1,1}^{\alpha}k_{\alpha}/k_1$ for all $\alpha\in I$ and $p_{\beta,1}^{1}=p_{1,1^*}^{\beta}k_{\beta}/k_1$ for all $\beta\in J$. By setting $i=j=l=1$ and $f=1^{*}$ in Lemma~\ref{jb} \ref{jb-4}, and using \eqref{2.1a}--\eqref{2.1c}, \ref{jiben-4} is also valid.
\end{proof}

\begin{lemma}\label{2.2}
For $x,z\in X$, we have $P_{1,1}(x,z)\times P_{1,1^{*}}(x,z)\subseteq\bigcup_{j\in I\cap J}R_{j}$.
\end{lemma}
\begin{proof}Pick $y\in P_{1,1}(x,z)$ and $y'\in P_{1,1^{*}}(x,z)$. Since $z\in P_{1,1}(y,y')$ and $x\in P_{1^{*},1}(y,y')$, we have $(y,y')\in R_{j}$ for some $j\in I\cap J$.
\end{proof}

In the following, we divide the proof of the first statement of Theorem~\ref{main1} into three subsections according to separate assumptions based on the cardinality of the set $I$.

\subsection{The case $|I|=1$}

By (\ref{1.1a}), we have $I=\{i\}$ for some $i\in\{1,1^{*},2\}$. In view of Lemma~\ref{jiben} \ref{jiben-2}, one gets $p_{1,1}^{i}k_i=k_1^2$. Since $k_1\geq k_2$, we obtain $p_{1,1}^i=k_i=k_1$. It follows from Lemma~\ref{jiben} \ref{jiben-1} that $p_{1,1^{*}}^{i}=0$. This implies $i\notin J$ and also $i\neq 1$ by Lemma \ref{3.2}. Thus $1\notin I$, and hence $\{1,1^{*},i,i^{*}\}\cap J=\emptyset$. In view of (\ref{set J}), one has $i=1^{*}$ and $J=\{0,2,2^{*}\}$. By Lemma~\ref{jiben} \ref{jiben-3}, we have $p_{1,1^{*}}^{2}=k_1(k_1-1)/(2k_2)$. Lemma~\ref{jiben} \ref{jiben-4} implies $(p_{1,1}^{1^{*}})^{2}k_1=k_1^2+2(p_{1,1^{*}}^{2})^{2}k_2.$ Substituting $p_{1,1}^{1^{*}}$ and $p_{1,1^{*}}^{2}$ into the above equation,  we get $k_2=(k_1-1)/2$, and so $p_{1,1^{*}}^{2}=p_{1,1^{*}}^{2^*}=k_1$. Since $p_{2,1}^{1}k_1=p_{1,1^*}^{2}k_2$ and $p_{2^*,1}^{1}k_1=p_{1,1^*}^{2^*}k_2$ from Lemma \ref{jb} \ref{jb-2}, one obtains $p_{2,1}^1=p_{2^*,1}^{1}=k_2$. Note that $k_2k_1=p_{2,1}^1k_1=p_{2^*,1}^1k_1$. Lemma \ref{jb} \ref{jb-1} implies $R_{2}R_1=R_{2^*}R_1=\{R_1\}$. It follows from Lemma \ref{d=4} that $d=4$.

\subsection{The case $|I|=2$}

By Lemmas \ref{1,1*} and \ref{k1>k2}, we only need to consider the case that $I=\{1,2\}$ or $I=\{1^*,2\}$, and $k_1=k_2$.

\textbf{Case 1.} $I=\{1,2\}$.

By Lemma~\ref{3.2} and \eqref{set J}, one gets $J=\{0,1,1^{*}\}$ or $J=\{0,1,1^{*},2,2^{*}\}$.
Since $k_{1}=k_2$, from Lemma~\ref{jiben} \ref{jiben-3} and \eqref{2.1a}, one has $2p_{1,1}^{1}+2p_{1,1^{*}}^{2}=k_1-1$. In view of Lemma~\ref{jiben} \ref{jiben-2}, we get $p_{1,1}^{1}+p_{1,1}^{2}=k_1$. By Lemma~\ref{jiben} \ref{jiben-4}, we obtain
\[(p_{1,1}^{1})^{2}+(p_{1,1}^{2})^{2}=k_1+2(p_{1,1^{*}}^{1})^{2}+2(p_{1,1^{*}}^{2})^{2}.\]
In view of (\ref{2.1a}),  we get
\[(k_1-p_{1,1}^{1})^{2}=k_1+(p_{1,1}^{1})^{2}+(k_1-1-2p_{1,1}^{1})^{2}/2.\]
Then $p_{1,1}^{1}=(k_1-1)/2$ and $p_{1,1^{*}}^{2}=(k_1-1-2p_{1,1}^{1})/2=0$. Thus, $J=\{0,1,1^{*}\}$.

Since $p_{1,1}^{1}\neq0$, there exist elements $x,y,z$ such that $(x,y),(y,z),(x,z)\in R_{1}$. Since $I\cap J=\{1\}$, from Lemma~\ref{2.2}, we have $P_{1,1^{*}}(x,z)=P_{1,1^{*}}(x,y)$, contrary to the fact that $z\in P_{1,1^{*}}(x,y)$.

\textbf{Case 2.} $I=\{1^{*},2\}$.

By Lemma~\ref{3.2} and \eqref{set J}, one has $J=\{0,2,2^{*}\}$. Since $k_1=k_2$, from Lemma~\ref{jiben} \ref{jiben-3}, we obtain $p_{1,1^{*}}^{2}=(k_1-1)/2$.

We claim $p_{1,1}^{2}=1$. Pick $(x,z)\in R_{2}$. Since $I\cap J=\{2\}$, Lemma~\ref{2.2} implies
\[P_{1,1^{*}}(x,z)\subseteq\bigcap_{y\in P_{1,1}(x,z)}P_{1,2^{*}}(x,y).\]
Since $p_{1,1^*}^2=p_{1,2^*}^1$ from Lemma \ref{jb} \ref{jb-2}, we have $|P_{1,1^{*}}(x,z)|=|P_{1,2^{*}}(x,y)|$ for all $y\in P_{1,1}(x,z)$, which implies $P_{1,1^{*}}(x,z)=P_{1,2^{*}}(x,y)$ for all $y\in P_{1,1}(x,z)$. Now suppose $y,y'\in P_{1,1}(x,z)$ are distinct. Since $z\in P_{1,1^{*}}(y,y')$ and $J=\{0,2,2^{*}\}$, we have $(y,y')\in R_{2}\cup R_{2^{*}}$. Then we may assume without loss of generality $(y,y')\in R_{2^{*}}$. However, this is a contradiction since $y\in P_{1,2^{*}}(x,y')=P_{1,2^{*}}(x,y)$.

By Lemma~\ref{jiben} \ref{jiben-2}, we have $p_{1,1}^{1^{*}}=k_1-1$. Lemma~\ref{jiben} \ref{jiben-4} implies
\[(p_{1,1}^{1^{*}})^{2}+(p_{1,1}^{2})^{2}=k_1+2(p_{1,1^{*}}^{2})^{2}.\]
Substituting $p_{1,1}^{1^{*}},p_{1,1}^{2}$ and $p_{1,1^{*}}^{2}$ into the above equation, one gets $k_1=3$.

By Lemma \ref{d=4}, it suffices to show that \eqref{eq-1} and \eqref{eq-2} are both valid.

Fix $(x,z)\in R_{2}$, and pick $y\in P_{1,1}(x,z)$. Let $w,w',w''$ be three elements such that $R_{1}(z)=\{w,w',w''\}$.  Since $p_{1,1}^{1^{*}}=2$, we may assume that $w',w''\in P_{1,1}(z,y)$.  In view of Lemma \ref{jiben} \ref{jiben-1} and (\ref{2.1c}), we get $p_{2,1}^{1}=p_{2^{*},1}^{1}=1$. Then we may assume $w'\in P_{2,1}(x,y)$ and $w''\in P_{2^{*},1}(x,y)$. By $p_{1,1^{*}}^{2}=1$, one has $w\in P_{1,1^{*}}(x,z)$. It follows from Lemma~\ref{R2} \ref{R2-1} that \eqref{eq-1} is valid.

Since $z\in P_{2,1}(x,w')$, one has $p_{2,1}^{2}\neq0$, which implies that there exists an element $y'\in P_{1,2}(x,z)$. By $k_1=3$, we get $R_{1}(x)=\{y,y',w\}$. Thus, \eqref{eq-2} follows from Lemma~\ref{R2} \ref{R2-2}.

\begin{remark}
Since $k_1=k_2=3$ and $d=4$, we have $|X|=13$. It follows from \cite[Result 1]{MH96} that $\mathfrak{X}$ is isomorphic to $\textrm{Cyc}(13,4)$.
\end{remark}

\subsection{The case $|I|=3$}

By Lemma \ref{k1>k2}, we only need to consider the case $k_1=k_2$. Lemma~\ref{3.2} implies $\{0,1,1^{*}\}\subseteq J$. In view of Lemma~\ref{jiben} \ref{jiben-4} and \eqref{2.1a}, we obtain
\begin{align}
(p_{1,1}^{1^{*}})^{2}+(p_{1,1}^{2})^{2}=k_1+(p_{1,1}^{1})^{2}+2(p_{1,1^{*}}^{2})^{2}.\label{J}
\end{align}

Suppose $J=\{0,1,1^{*}\}$. In view of Lemma~\ref{jiben} \ref{jiben-3} and \eqref{2.1a}, we get $p_{1,1}^{1}=(k_1-1)/2$. Then Lemma~\ref{jiben} \ref{jiben-2} implies $p_{1,1}^{1^{*}}+p_{1,1}^{2}=(k_1+1)/2$. Since $p_{1,1^{*}}^{2}=0$, from \eqref{J}, one gets $(p_{1,1}^{1^{*}})^{2}+(p_{1,1}^{2})^{2}=(k_1+1)^{2}/4$. It follows that $p_{1,1}^{1^{*}}p_{1,1}^{2}=0$, a contradiction. Hence, $J=\{0,1,1^{*},2,2^{*}\}$.

Suppose $p_{1,1}^{2}=p_{1,1^{*}}^{2}$.  In view of Lemma~\ref{jiben} \ref{jiben-3} and (\ref{2.1a}), one gets $p_{1,1}^{1}+p_{1,1}^{2}=(k_1-1)/2$. By Lemma~\ref{jiben} \ref{jiben-2}, we have  $p_{1,1}^{1^{*}}=(k_1+1)/2$. In view of \eqref{J}, we get $(p_{1,1}^{1})^{2}+(p_{1,1}^{2})^{2}=(k_1-1)^{2}/4$. It follows that $p_{1,1}^{1}p_{1,1}^{2}=0$, a contradiction. Hence, $p_{1,1}^{2}\neq p_{1,1^{*}}^{2}$.

By Lemma \ref{d=4}, it suffices to show that \eqref{eq-1} and \eqref{eq-2} are both valid.

Fix $(x,z)\in R_{2}$, and pick an element $w\in R_{1}(z)$. Then $P_{1,1}(x,z)\subseteq(R_{1}\cup R_{1^{*}}\cup R_{2^{*}})(w)$ by \eqref{1.1a}. Suppose first $P_{1,1}(x,z)\cap(R_{1}\cup R_{1^{*}})(w)\neq\emptyset$. Then $(x,w)\in R_{1}\cup R_{1^{*}}\cup R_{2}\cup R_{2^{*}}$ by \eqref{1.1a}, \eqref{1.1b} and \eqref{1.1c}. Next suppose $P_{1,1}(x,z)\subseteq R_{2^{*}}(w)$. Since $p_{1,1}^2=p_{1^*,2}^1$ from Lemma \ref{jb} \ref{jb-2}, we have $P_{1,1}(x,z)=P_{1^{*},2}(z,w)$. Since $p_{1^{*},1}^{2}=p_{1,1^{*}}^{2}\neq0$ from (\ref{2.1b}), there exists an element $y_{0}\in P_{1^{*},1}(x,z)$. By $y_{0}\notin P_{1^{*},2}(z,w)$, we get $(y_{0},w)\in R_{1}\cup R_{1^{*}}$, which implies $(x,w)\in R_{1}\cup R_{1^{*}}\cup R_{2}\cup R_{2^{*}}$. Therefore, \eqref{eq-1} follows from Lemma~\ref{R2} \ref{R2-1}. Note that \eqref{eq-1} implies
\begin{align}
R_{1^{*}}R_{2^{*}}\subseteq\{R_{1},R_{1^{*}},R_{2},R_{2^{*}}\}.\label{subsec3-3}
\end{align}

Next, pick an element $w'\in R_{1}(x)$. Then $P_{1,1}(x,z)\subseteq(R_{0}\cup R_{1}\cup R_{1^{*}}\cup R_{2}\cup R_{2^{*}})(w')$ by \eqref{1.1b}. Suppose first $P_{1,1}(x,z)\cap(R_{0}\cup R_{1}\cup R_{1^{*}}\cup R_{2})(w')\neq\emptyset$. By \eqref{1.1a},\eqref{1.1c} and \eqref{subsec3-3}, we have $(z,w')\in R_{1}\cup R_{1^{*}}\cup R_{2}\cup R_{2^{*}}$. Now suppose $P_{1,1}(x,z)\subseteq P_{1,2}(x,w')$. Since $p_{1^*,2}^1=p_{1,1}^{2}\neq p_{1,1^{*}}^{2}=p_{1,2}^1$ from Lemma \ref{jb} \ref{jb-2}, one gets $p_{1^{*},2}^{1}<p_{1^{*},1}^{2}=p_{1,1^*}^{2}$. Pick an element $y\in P_{1,1}(x,z)$. Since $(y,w')\in R_{2}$, there exists an element $x'\in P_{1^{*},1}(y,w')$ such that $x'\notin P_{1^{*},2}(y,z)$. Then $(x',z)\in R_{1}\cup R_{1^{*}}$  by \eqref{1.1a}. \eqref{1.1b} and \eqref{1.1c} imply $(z,w')\in R_{1}\cup R_{1^{*}}\cup R_{2}\cup R_{2^{*}}$.  Thus, \eqref{eq-2} follows from Lemma~\ref{R2} \ref{R2-2}.

\begin{remark}
We do not know any example occurring in the case $|I|=3$ with $k_1=k_2$.  If there is an association scheme under this case, $\mathfrak{X}$ will not be pseudocyclic by Lemma~\ref{skew}. Indeed $\mathfrak{X}={\rm Cyc}(13,4)$ satisfies $|I|=2$.
\end{remark}

\begin{remark}
In Theorem \ref{main1}, the assumption $k_1>1$ is necessary. Indeed, the group scheme (see {\rm\cite[Chapter II, Example 2.1 (2)]{EB84}} for a definition of a group scheme) over a cyclic group of order $n>2$ satisfies all the conditions of Theorem \ref{main1} and $d=n-1$.
\end{remark}

\section{Schur rings}

Let $G$ be a finite, multiplicatively written, group with identity $e$. For a set $X\subseteq G$ and an integer $m$, we set $X^{(m)}=\{x^m\mid x\in X\}$. The formal sum of the elements of $X$ is denoted by $\underline{X}$, and we treat this element as an element of the group ring $\mathbb{Z}[G]$. Given $\mathcal{S}\subseteq2^G$, we set $\underline{\mathcal{S}}=\{\underline{X}\mid X\in\mathcal{S}\}$.

A ring $\mathcal{A}\subseteq\mathbb{Z}[G]$ is called a {\em Schur ring} (briefly an {\em S-ring}) over the group $G$ if there exists a
partition $\mathcal{S}=\mathcal{S}(\mathcal{A})$ of $G$ satisfying the following conditions:
\begin{enumerate}
\item\label{S-ring-1} the set $\underline{\mathcal{S}}$ is a linear basis of $\mathcal{A}$;

\item\label{S-ring-2} $\{e\}\in\mathcal{S}$;

\item\label{S-ring-3} $X^{(-1)}\in\mathcal{S}$ for all $X\in\mathcal{S}$.
\end{enumerate}

It can be proved that a $\mathbb{Z}$-submodule $\mathcal{A}\subseteq\mathbb{Z}[G]$ is an S-ring over $G$ if and only if the following two conditions are satisfied:
\begin{itemize}
\item $\mathcal{A}$ is closed with respect to the operation $\sum_{x}a_{x}x\mapsto\sum_{x}a_{x}x^{-1}$;\vspace{-0.2cm}\\

\item $\mathcal{A}$ is a ring with identity $\underline{\{e\}}$ (resp. $\underline{G}$) with respect to the ordinary (resp. the componentwise) multiplication.
\end{itemize}

A partition $\mathcal{S}$ of $G$ is called a {\em Schur partition} if it is satisfies conditions \ref{S-ring-2}, \ref{S-ring-3} and $\mathcal{A}:=\langle\underline{\mathcal{S}}\rangle$ is a
subalgebra of $\mathbb{Z}[G]$. One can see that there is a 1-1 correspondence between S-rings over $G$ and Schur partitions of $G$. The elements (classes) of $\mathcal{S}$ are called the {\em basic sets} of $\mathcal{A}$; the number ${\rm rk}(\mathcal{A}):=|\mathcal{S}|$ is called the {\em rank} of $\mathcal{A}$.

Any union of basic sets of the S-ring $\mathcal{A}$ is called an {\em $\mathcal{A}$-set}. Thus $X\subseteq G$ is an $\mathcal{A}$-set if and only if $\underline{X}\in\mathcal{A}$. A subgroup $H$ of $G$ is called an {\em $\mathcal{A}$-subgroup} if $H$ is an $\mathcal{A}$-set. We say that $\mathcal{A}$ is {\em imprimitive} if there is a proper non-trivial $\mathcal{A}$-subgroup of $G$; otherwise $\mathcal{A}$ is {\em primitive}. If $\mathcal{S}(\mathcal{A})$ coincides with the set of all orbits of a group $K\leq {\rm Aut}(G)$ on $G$, we say that $\mathcal{A}$ is an {\em orbit} S-ring and denoted by $\mathcal{O}(K,G)$.

Let $(X,\{R_{i}\}_{i=0}^{d})$ be an association scheme. For an element $x$ of $X$ and a nonempty subset $F$ of $\{R_{i}\}_{i=0}^{d}$, define $F(x)=\{y\in X\mid (x,y)\in \bigcup_{f\in F}f\}$. A connection between S-rings and association schemes is as follows.

\begin{thm}[{\cite[Section 3]{BX17}}]\label{s-ring and a.s}
Let $\mathcal{S}$ be a partition of $G$ and let $R_G(\mathcal{S})=\{R_G(X)\mid X\in\mathcal{S}\}$, where $R_G(X)=\{(g,xg)\mid g\in G,x\in X\}$ for any $X\in\mathcal{S}$. Then $\mathcal{S}$ is a Schur partitions of $G$ if and only if the pair $(G,R_G(\mathcal{S}))$ is an association scheme. Moreover, in this case, there is a one-to-one correspondence between closed subsets of $R_G(S)$ and $\mathcal{A}$-subgroups given by $F$ to $F(e)$.
\end{thm}

Let $\Gamma$ denote a Cayley digraph ${\rm Cay}(G,S)$, where $G$ is a finite, multiplicatively written, group with identity $e$. Write $N_{\tilde{i}}$ instead of $\Gamma_{\tilde{i}}(e)$ for all $\tilde{i}\in\tilde{\partial}(\Gamma)$. We say that the $\mathbb{Z}$-submodule of the group algebra $\mathbb{Z}[G]$ spanned by all the elements $\underline{N_{\tilde{i}}}$ with $\tilde{i}\in\tilde{\partial}(\Gamma)$ is the {\em two-way distance module} of $\Gamma$, and is denoted by $\mathcal{D}(G,S)$.

\begin{prop}\label{wdrdg s-ring}
The Cayley digraph ${\rm Cay}(G,S)$ is weakly distance-regular if and only if the two-way distance
module $\mathcal{D}(G,S)$ is an S-ring.
\end{prop}
\begin{proof} In the notation of Theorem \ref{s-ring and a.s}, we have $R_G(N_{\tilde{i}})=\Gamma_{\tilde{i}}$ for all $\tilde{i}\in\tilde{\partial}(\Gamma)$. The desired result follows from Theorem \ref{s-ring and a.s}.\end{proof}

The connection between primitive S-rings and primitive weakly distance-regular Cayley digraphs is established through the following simple observation.

\begin{prop}\label{primitive}
Let ${\rm Cay}(G,S)$ be a weakly distance-regular digraph with $\mathcal{D}(G,S)$ as its two-way distance module. Then $\mathcal{D}(G,S)$ is a primitive S-ring if and only if ${\rm Cay}(G,S)$ is primitive.
\end{prop}
\begin{proof} The weakly distance-regular digraph $\Gamma:={\rm Cay}(G,S)$ is primitive if and only if each digraph $(G,\Gamma_{\tilde{i}})$ with $\tilde{i}\in\tilde{\partial}(\Gamma)\setminus\{(0,0)\}$ is strongly connected. Since the digraph $(G,\Gamma_{\tilde{i}})$ is isomorphic to the digraph Cay$(G,N_{\tilde{i}})$, $\Gamma$ is primitive if and only if the set $N_{\tilde{i}}$ generates the group $G$ for each $\tilde{i}\in\tilde{\partial}(\Gamma)\setminus\{(0,0)\}$ . The desired result follows.\end{proof}

Let $\mathcal{A}$ be an S-ring over a group $G$, and $L$ be a normal $\mathcal{A}$-subgroup. One can define a quotient S-ring over the factor group $G/L$ as follows.
The natural homomorphism $\pi:G\rightarrow G/L$ can be canonically extended to a ring
homomorphism $\mathbb{Z}[G]\rightarrow\mathbb{Z}[G/L]$ which we also denote by $\pi$. Then $\pi(\mathcal{A})$ is a subring of $\mathbb{Z}[G/L]$. By \cite[Proposition 4.8]{MM04}, $\pi(\mathcal{A})$ is an S-ring over $G/L$ the basic sets of which are given by $\mathcal{S}(\mathcal{A}/L)=\{X/L\mid X\in\mathcal{S}(\mathcal{A})\}$, where $X/L=\{xL\mid x\in X\}$. We call this $\pi(\mathcal{A})$ the {\em quotient S-ring} and denote it by $\mathcal{A}/L$.

Let $\mathcal{A}$ be an S-ring over a group $G$. Given $X\in\mathcal{S}(\mathcal{A})$ the set
\begin{align}
{\rm rad}(X)=\{g\in G\mid gX=Xg=X\}\nonumber
\end{align}
is obviously a subgroup of $G$ called the {\em radical} of $X$. Then ${\rm rad}(X)$ is the largest subgroup of $G$ such that $X$ is a union of left as well as right cosets by it. By \cite[Proposition 23.5]{HW64}, ${\rm rad}(X)$ is an $\mathcal{A}$-subgroup.

We now apply the above results to the case where $G$ is an abelian group, change multiplicative
notations to additive ones, and denote the identity element by $0$ instead of $e$. The following theorem forms the basis of the Schur theory of S-rings over abelian groups. In \cite{AEB98} it is called the multiplier theorem.

\begin{thm}\label{Multiplier1}
Let $\mathcal{A}$ be an S-ring over an abelian group $G$, $X\in\mathcal{S}(\mathcal{A})$, and $m$ an integer coprime with the order of $G$. Then $X^{(m)}\in\mathcal{S}(\mathcal{A})$.
\end{thm}

The following statements are two properties of basic sets of
S-rings over cyclic groups that will be necessary.


\begin{lemma}\label{genrates}
Let $X$ be a basic set of an S-ring $\mathcal{A}$ over a cyclic group $G$. If $X$ generates $G$, then $X$ contains a generator of $G$.
\end{lemma}
\begin{proof}
Since $G$ is cyclic, $G\cong\mathbb{Z}_n$ and $G/{\rm rad}(X)\cong\mathbb{Z}_m$ for some $m,n$ with $m\mid n$. By \cite[Theorem 5.9 (2)--(4)]{MM04},  there exists a generator $g+\rad(X)$ of $G/{\rm rad}(X)$ with $g\in X$. Since the ring homomorphism $\mathbb{Z}/n\mathbb{Z}$ to $\mathbb{Z}/m\mathbb{Z}$ restricts to a surjection from $(\mathbb{Z}/n\mathbb{Z})^{\times}$ to $(\mathbb{Z}/m\mathbb{Z})^{\times}$, there exists a generator of $G$ in $g+\rad(X)\subseteq X$.
\end{proof}

\begin{lemma}\label{q=3 step1-1}
Let $\mathcal{A}$ be an orbit S-ring over a cyclic group $G$. Suppose that there exists a basic set $X$ of $\mathcal{A}$ such that $X^{(-1)}\subseteq X+X:=\{x+x'\mid x,x'\in X\}$ and that $X$ generates $G$. Then there exists an integer $b$ such that $Y^{(b)}=Y^{(-b-1)}=Y$ for all $Y\in\mathcal{S}(\mathcal{A})$. In particular, $Y^{(-1)}\subseteq Y+Y$.
\end{lemma}
\begin{proof}
Since $X$ generates $G$, from Lemma \ref{genrates}, $X$ contains a generator $x$ of $G$. Since $-x\in X^{(-1)}\subseteq X+X$, there exist $x',x''\in X$ such that $-x=x'+x''$. Write $x'=bx$ for some integer $b$. Then $x''=-(b+1)x$. Since $\mathcal{A}$ is an orbit S-ring, there exists a subgroup $K\leq {\rm Aut}(G)$ such that every $Y\in \mathcal{S}(\mathcal{A})$ is a $K$-orbit. In particular, since $X$ is a $K$-orbit, there exists $f\in K$ such that $bx=f(x)$. Then $Y^{(b)}=f(Y)=Y$, and similarly $Y^{(-b-1)}=Y$, for all $Y\in \mathcal{S}(\mathcal{A})$. The first statement is valid.

For each $y\in Y$, since $by,-(b+1)y\in Y$, we have $-y=by-(b+1)y\in Y+Y$. The second statement is also valid.
\end{proof}

Let $\mathcal{A}$ be an S-ring over a cyclic group $G$. Then the group ${\rm Aut}(G)$ acts transitively on the set of its generators. Theorem \ref{Multiplier1} and Lemma \ref{genrates} imply that the group ${\rm rad}(X)$ does not depend on the choice of $X\in\mathcal{S}(\mathcal{A})$ such that $X$ generates $G$. We call this group the {\em radical} of $\mathcal{A}$ and denote it by ${\rm rad}(\mathcal{A})$. The S-ring $\mathcal{A}$ is {\em free} if ${\rm rad}(\mathcal{A})=\{0\}$. It is easy to see that the quotient S-ring $\mathcal{A}/{\rm rad}(\mathcal{A})$ is free.

\begin{lemma}\label{p^2}
For an odd prime $p$ and a positive integer $m$, let $G$ be a cyclic group of group $p^m$ and $\mathcal{A}:=\mathcal{O}(K,G)$ an orbit S-ring over $G$ for some $K\leq {\rm Aut}(G)$. Suppose $p\mid |K|$. Then $G^{(p^{m-1})}\leq {\rm rad}(\mathcal{A})$. In particular, $\mathcal{A}$ is not free.
\end{lemma}
\begin{proof}
Let $\sigma$ be an automorphism of $G$ defined by $\sigma(x)=(1+p^{m-1})x$ for $x\in G$. Then the order of $\sigma$ is $p$. Since ${\rm Aut}(G)\cong \mathbb{Z}_{p^{m-1}(p-1)}$ from \cite[Chapter 4, Theorem 2]{IR}, ${\rm Aut}(G)$ has $\langle \sigma\rangle$ as its unique subgroup of order $p$. It follows that $\langle \sigma\rangle\leq K$.

Let $X$ be a basic set of $\mathcal{A}$ containing a generator $G$. Pick $x\in X$. Since $\mathcal{A}$ is an orbit S-ring,
\begin{align}
X\supseteq\{\sigma^k(x)\mid k\in\mathbb{Z}\}=\{(1+p^{m-1})^kx\mid k\in\mathbb{Z}\}=\{x+p^{m-1}kx\mid k\in\mathbb{Z}\}.\nonumber
\end{align}
Since $x$ is a generator, we have $X\supseteq x+G^{(p^{m-1})}$. Since $x\in X$ was arbitrary, $G^{(p^{m-1})}\leq{\rm rad}(X)={\rm rad}(\mathcal{A})$. Thus, $\mathcal{A}$ is not free.
\end{proof}

We recall the definition of 
a tensor product of two S-rings. Let $\mathcal{A}_{1}$ and $\mathcal{A}_{2}$ be S-rings over groups $G_1$ and $G_2$, respectively. Then the {\em tensor product} $\mathcal{A}_{1}\otimes\mathcal{A}_{2}$ is obviously an S-ring over the group $G_1\times G_{2}$ with $\mathcal{S}(\mathcal{A}_{1}\otimes\mathcal{A}_{2})=\{X_{1}\times X_{2}\mid X_{1}\in\mathcal{S}(\mathcal{A}_{1}),~X_{2}\in\mathcal{S}(\mathcal{A}_{2})\}$. The tensor product of more than two S-rings can be defined similarly.

\begin{prop}\label{jb3-2}
Let $\mathcal{A}$ be a free S-ring over a cyclic group $G$. Then there exist subgroups $U_{0},U_{1},\ldots,U_{k}\leq G$ with $k\geq0$ such that
\begin{align}
G=U_0\times U_1\times\cdots\times U_k,~{\rm and}~\mathcal{A}=\mathcal{A}_{0}\otimes\mathcal{A}_{1}\otimes\cdots\otimes\mathcal{A}_{k},\nonumber
\end{align}
where $\mathcal{A}_{0}$ is a free orbit S-ring over the group $U_{0}$ and $\mathcal{A}_i$ is an S-ring of rank $2$ over the group $U_{i}$ for all $i>0$. Moreover, the orders $|U_{i}|$, $i\in\{0,1,\ldots,k\}$, are pairwise coprime  and $|U_{i}|>3$ for all $i>0$.
\end{prop}
\begin{proof}
By \cite[Theorems 4.6 and 4.10]{MM09}, the first statement is valid. The second statement is also valid from \cite[Theorem 6]{IK04}.
\end{proof}

Let $\mathcal{A}$ be an S-ring over $\mathbb{Z}_n$, and $\pi$ be the natural mapping from $\mathbb{Z}_{n}$ to $\mathbb{Z}_{n}/L$, where $L={\rm rad}(\mathcal{A})$. Then $\mathcal{A}':=\mathcal{A}/L$ is a free S-ring over $\mathbb{Z}_{n}/L$. By the definition of quotient S-ring, $\pi(X)$ is a basic set of $\mathcal{A}'$ for all $X\in\mathcal{S}(\mathcal{A})$. By Proposition \ref{jb3-2}, there exist subgroups $L\leq U_{0},U_{1},\ldots,U_{k}\leq \mathbb{Z}_{n}$ with $k\geq0$ such that
\begin{align}
\mathbb{Z}_{n}/L&=U_0/L\times U_1/L\times\cdots\times U_k/L,\label{G/L}\\
\mathcal{A}'&=\mathcal{A}_0\otimes\mathcal{A}_1\otimes\cdots\otimes\mathcal{A}_k,\label{A_G/L}
\end{align}
where $\mathcal{A}_0$ is a free orbit S-ring over the group $U_{0}/L$,
\begin{align}
&\text{$\mathcal{A}_i$ {is an S-ring of rank} $2$ over the group $U_{i}/L$  for all $i>0$},\label{rank 2}\\
&\text{$|U_{i}/L|$, $i\in\{0,1,\ldots,k\}$, are~pairwise~coprime, and $|U_{i}/L|>3$ for all $i>0$}.\label{coprime}
\end{align}
Denote by $\varphi_i$ the projection from $\mathbb{Z}_{n}/L$ to $U_{i}/L$ for $0\leq i\leq k$. Since $\pi(X)\in\mathcal{S}(\mathcal{A}')$ for all $X\in\mathcal{S}(\mathcal{A})$, by the definition of tensor product, $\varphi_i\pi(X)$ is a basic set of $\mathcal{A}_i$ for $0\leq i\leq k$.

We prove the following results under the above assumptions.

\begin{lemma}\label{basic sets}
Let $L_0\in U_0/L$ and $L_i\in U_i/L\setminus\{L\}$ for $1\leq i\leq k$. Then $(L_0,\underbrace{L,\ldots,L}_k)$ and $(L_0,L_1,\ldots,L_k)$ are not contained in the same basic set of $\mathcal{A}'$.
\end{lemma}
\begin{proof}
This is immediate from \eqref{A_G/L}.
\end{proof}

\begin{lemma}\label{isomorphic}
Let $X\in\mathcal{S}(\mathcal{A})$. Suppose that $X$ generates $\mathbb{Z}_n$. The following hold:
\begin{enumerate}
\item\label{isomorphic-1} $X$ is a union of $L$-cosets;

\item\label{isomorphic-2} $\pi(X)=\varphi_0\pi(X)\times(U_1/L\setminus\{L\})\times\cdots\times(U_k/L\setminus\{L\})$. In particular, each element in $\varphi_0\pi(X)$ is a generator of the group $U_0/L$.
\end{enumerate}
\end{lemma}
\begin{proof}
\ref{isomorphic-1}~Since ${\rm rad}(X)=L$, $X$ is a union of $L$-cosets.

\ref{isomorphic-2}~Since $\pi(X)\in\mathcal{S}(\mathcal{A}')$, \eqref{A_G/L} implies $\pi(X)=\varphi_0\pi(X)\times\cdots\times\varphi_k\pi(X)$.
By \eqref{rank 2}, one gets $\varphi_i\pi(X)=\{L\}$ or $U_i/L\setminus\{L\}$. Since $\pi(X)$ generates $\mathbb{Z}_n/L$, we obtain $\varphi_i\pi(X)=U_i/L\setminus\{L\}$. Thus, the first statement is valid. Since $\varphi_0\pi(X)$ is a basic set of $\mathcal{A}_0$ which is an orbit S-ring over the cyclic group $U_0/L$, the second statement follows.
\end{proof}

\begin{lemma}\label{(1,q-1) nonsymmetric}
Let $X$ be a basic set of $\mathcal{A}$ such that $X$ generates $\mathbb{Z}_n$ and $X\cap X^{(-1)}=\emptyset$. Then $\varphi_0\pi(X)\cap\varphi_0\pi(X)^{(-1)}=\emptyset$.
\end{lemma}
\begin{proof}
By Lemma \ref{isomorphic} \ref{isomorphic-1}, $X$ is a union of $L$-cosets. The fact $X\cap X^{(-1)}=\emptyset$ implies $\pi(X)\cap\pi(X)^{(-1)}=\emptyset$. By Lemma \ref{isomorphic} \ref{isomorphic-2}, we get $\pi(X)^{(\pm1)}=\varphi_0\pi(X)^{(\pm1)}\times(U_1/L\setminus\{L\})\times\cdots\times(U_k/L\setminus\{L\})$. Thus, the desired result follows.
\end{proof}

\begin{lemma}\label{jb4}
Let $X,Y\in\mathcal{S}(\mathcal{A})$. Suppose that $X$ generates $\mathbb{Z}_n$. Then $|X|\geq |Y|$.
\end{lemma}
\begin{proof}
Since $\mathcal{A}_0$ is an orbit S-ring, from Lemma \ref{isomorphic} \ref{isomorphic-2}, we have $|\varphi_0\pi(X)|\geq|\varphi_0\pi(Y)|$. Since $\varphi_i\pi(Y)\in\mathcal{S}(\mathcal{A}_i)$ for $1\leq i\leq k$, from \eqref{rank 2} and Lemma \ref{isomorphic} \ref{isomorphic-2} again, one has
\begin{align}
|\pi(X)|=|\varphi_0\pi(X)|\prod_{i=1}^k(|U_i/L|-1)\geq\prod_{i=0}^k|\varphi_i\pi(Y)|=|\pi(Y)|.\nonumber
\end{align}
Since $X$ is a union of $L$-cosets from Lemma \ref{isomorphic} \ref{isomorphic-1}, we get $|X|/|L|=|\pi(X)|$, while $|\pi(Y)|\geq|Y|/|L|$. The desired inequality now follows.
\end{proof}

\begin{lemma}\label{pre-claim}
Let $X,Y\in\mathcal{S}(\mathcal{A})$. Suppose $\varphi_0\pi(Y)\subseteq\varphi_0\pi(X)\pm\varphi_0\pi(X)$. If  $\varphi_i\pi(Y)\subseteq\langle\varphi_i\pi(X)\rangle$ for $1\leq i\leq k$, then $\pi(Y)\subseteq\pi(X)\pm\pi(X)$.
\end{lemma}
\begin{proof}
Since $\pi(Y)=\varphi_0\pi(Y)\times\cdots\times\varphi_k\pi(Y)$ and $\pi(X)=\varphi_0\pi(X)\times\cdots\times\varphi_k\pi(X)$,  it suffices to show that $\varphi_i\pi(Y)\subseteq\varphi_i\pi(X)\pm\varphi_i\pi(X)$ for $1\leq i\leq k$. By \eqref{rank 2}, we have  $\varphi_i\pi(X)=\{L\}$ or $U_i/L\setminus\{L\}$, which implies $\varphi_i\pi(X)\pm\varphi_i\pi(X)=\{L\}$ or $U_i/L$. Since $\varphi_i\pi(Y)\subseteq\langle\varphi_i\pi(X)\rangle$, the result follows.
\end{proof}

\begin{lemma}\label{q=3 step1}
Suppose that there exists a basic set $X$ of $\mathcal{A}$ such that $X^{(-1)}\subseteq X+X$ and that $X$ generates $\mathbb{Z}_n$. Then $\pi(Y)^{(-1)}\subseteq \pi(Y)+\pi(Y)$ for all $Y\in\mathcal{S}(\mathcal{A})$. In particular, if $Y$ generates $\mathbb{Z}_n$, then $Y^{(-1)}\subseteq Y+Y$.
\end{lemma}
\begin{proof}
By the assumption, we have $\varphi_0\pi(X)^{(-1)}\subseteq\varphi_0\pi(X)+\varphi_0\pi(X)$. Since $\mathcal{A}_0$ is an orbit S-ring over the cyclic group $U_0/L$, from Lemma \ref{q=3 step1-1}, we have $\varphi_0\pi(Y^{(-1)})=\varphi_0\pi(Y)^{(-1)}\subseteq\varphi_0\pi(Y)+\varphi_0\pi(Y)$. By \eqref{rank 2}, one gets $\varphi_i\pi(Y^{(-1)})=\varphi_i\pi(Y)\subseteq\langle\varphi_i\pi(Y)\rangle$ for $1\leq i\leq k$. By Lemma \ref{pre-claim}, the first statement is valid.

If $Y$ generates $\mathbb{Z}_n$, from Lemma \ref{isomorphic} \ref{isomorphic-1}, then $Y$ is a union of $L$-cosets, in other words, $\pi^{-1}(\pi(Y))=Y$. This implies that the second statement is valid.
\end{proof}

\section{The subdigraph $\Delta_q$}

In the remainder of this paper, $\Gamma$ always denotes a weakly distance-regular circulant. Let $T$ be the set of integers $q$ such that $(1,q-1)$ is a type of $\Gamma$. That is,
\begin{align}
T=\{q\mid (1,q-1)\in\tilde{\partial}(\Gamma)\}.\nonumber
\end{align}
For $q\in T$, let $F_{q}$ be the smallest closed subset containing $\Gamma_{1,q-1}$, and $\Delta_{q}$ be the digraph with the vertex set $F_{q}(0)$ and arc set $\Gamma_{1,q-1}\cap(F_{q}(0))^2$. Note that $\Delta_q=\Gamma$ if $T=\{q\}$, but in general, $\Delta_q$ is not weakly distance-regular.

Under the assumption of Theorem \ref{Main2}, $\Gamma$ has only one type of arcs, and hence $|T|=1$. However, we not only consider weakly distance-regular circulants of one type of arcs, but also discuss the subdigraph $\Delta_{q}$ for $q>3$ in Proposition \ref{(1,q-1)} and the subdigraph $\Delta_{3}$ under the assumption $\min T=3$ in Proposition \ref{q=3}, which will be used frequently in the forthcoming paper \cite{AM} to classify weakly distance-regular circulants of more than one types of arcs. Also in \cite{AM}, we will consider $\Delta_{3}$ under the assumption $\{2,3\}\subseteq T$.

Note that the minimum of $T$ is the girth of $\Gamma$ and is at least $2$. The girth can be $2$, for example, for $\Gamma={\rm Cay}(\mathbb{Z}_4,\{1,2\})$. If the girth is $2$, then $T$ must contain $\{2\}$ properly, since $\Gamma$ is assumed not to be undirected. Thus, under the assumption of Theorem \ref{Main2}, we have $T=\{q\}$ with $q\geq3$.

The following lemma is proved in \cite{YYF22} under the assumption that $\Gamma$ is thick (see \cite{YYF22} for a definition), but its proof does not use this assumption.

\begin{lemma}\label{pure}
Let $q\in T$. If $p_{(1,q-1),(1,q-1)}^{(2,q-2)}=k_{1,q-1}$, then $\Gamma_{1,q-1}^{l}=\{\Gamma_{l,q-l}\}$ for $1\leq l\leq q-1$ and $\Delta_{q}$ is isomorphic to $C_q[\overline{K}_m]$, where $m=k_{1,q-1}$.
\end{lemma}

Fix $q\in T$. By the second statement of Theorem \ref{s-ring and a.s}, we may assume $F_{q}(0)=\mathbb{Z}_n$ for some positive integer $n$.  In view of \cite[Proposition 1.5.1]{PHZ96}, the configuration $(\mathbb{Z}_n,\{\Gamma_{\tilde{i}}\cap\mathbb{Z}_n^2\mid \Gamma_{\tilde{i}}\in F_{q}\})$ is an association scheme. Theorem \ref{s-ring and a.s} implies that there exists an S-ring $\mathcal{A}$ over the group $\mathbb{Z}_n$ with Schur partition $\{N_{\tilde{i}}\mid \Gamma_{\tilde{i}}\in F_{q}\}$, where $N_{\tilde{i}}=\Gamma_{\tilde{i}}(0)$.

In the notations introduced after Proposition \ref{jb3-2}, we write $\overline{N}_{\tilde{i}}$ instead of $\varphi_0\pi(N_{\tilde{i}})$ for all $N_{\tilde{i}}\in\mathcal{S}(\mathcal{A})$. Note that $\mathcal{A}_0$ is an S-ring over the group $U_0/L$ with $\mathcal{S}(\mathcal{A}_0)=\{\overline{N}_{\tilde{i}}\mid N_{\tilde{i}}\in\mathcal{S}(\mathcal{A})\}$. Then Theorem \ref{s-ring and a.s} implies that $(U_{0}/L,\{R_{\tilde{i}}\}_{N_{\tilde{i}}\in\mathcal{S}(\mathcal{A})})$ is an association scheme, where $R_{\tilde{i}}=R_{U_{0}/L}(\overline{N}_{\tilde{i}})$. Note that this expression of the set of relations may contain duplicates; it may happen that $R_{\tilde{i}}=R_{\tilde{j}}$ for distinct $N_{\tilde{i}},N_{\tilde{j}}\in \mathcal{S}(\mathcal{A})$.

\begin{lemma}\label{genrators}
We have $\pi(N_{1,q-1})=\overline{N}_{1,q-1}\times (U_1/L\setminus\{L\})\times\cdots\times (U_k/L\setminus\{L\})$, where $|U_{i}/L|$, $i\in\{0,1,\ldots,k\}$, are pairwise coprime, and $|U_{i}/L|>3$ for all $i>0$. In particular,
\begin{enumerate}
\item\label{genrators-1} ${\rm Cay}(\mathbb{Z}_n/L,\pi(N_{1,q-1}))\cong{\rm Cay}(U_0/L,\overline{N}_{1,q-1})\times K_{|U_1/L|}\times\cdots\times K_{|U_k/L|}$;

\item\label{genrators-2} ${\rm Cay}(\mathbb{Z}_n,N_{1,q-1})\cong({\rm Cay}(U_0/L,\overline{N}_{1,q-1})\times K_{|U_1/L|}\times\cdots\times K_{|U_k/L|})[\overline{K}_{|L|}]$;

\item\label{genrators-3} each element in $\overline{N}_{1,q-1}$ is a generator of the group $U_0/L$;

\item\label{genrators-4} $N_{1,q-1}$ contains a generator of $\mathbb{Z}_n$;

\item\label{genrators-5} $R_{1,q-1}$ is non-symmetric.
\end{enumerate}
\end{lemma}
\begin{proof}
Since $N_{1,q-1}$ generates $\mathbb{Z}_n$, from Lemma \ref{isomorphic} \ref{isomorphic-2} and \eqref{coprime}, the first statement is valid. It follows that ${\rm Cay}(\mathbb{Z}_n/L,\pi(N_{1,q-1}))={\rm Cay}(U_0/L,\overline{N}_{1,q-1})\times {\rm Cay}(U_1/L,U_1/L\setminus\{L\})\times\cdots\times {\rm Cay}(U_k/L,U_1/L\setminus\{L\})$. Thus, \ref{genrators-1} holds.  \ref{genrators-2} is valid from Lemma \ref{isomorphic} \ref{isomorphic-1}. \ref{genrators-3} is immediate from Lemma \ref{isomorphic} \ref{isomorphic-2}. \ref{genrators-4} is valid from Lemma \ref{genrates}. And \ref{genrators-5} follows from Lemma \ref{(1,q-1) nonsymmetric}.
\end{proof}

\begin{lemma}\label{claim}
Let $\Gamma_{\tilde{i}}\in F_q$. The following hold:
\begin{enumerate}
\item\label{claim-1} $\Gamma_{\tilde{i}}\in\Gamma_{1,q-1}^2$ if and only if $R_{\tilde{i}}\in R_{1,q-1}^2$;

\item\label{claim-2} $\Gamma_{\tilde{i}}\in\Gamma_{1,q-1}\Gamma_{q-1,1}$ if and only if $R_{\tilde{i}}\in R_{1,q-1}R_{q-1,1}$.
\end{enumerate}
\end{lemma}
\begin{proof}
Note that $\Gamma_{\tilde{i}}\in\Gamma_{1,q-1}^2$ is equivalent to $N_{\tilde{i}}\subseteq N_{1,q-1}+N_{1,q-1}$, and $R_{\tilde{i}}\in R_{1,q-1}^2$ is equivalent to $\overline{N}_{\tilde{i}}\subseteq\overline{N}_{1,q-1}+\overline{N}_{1,q-1}$. The necessity is trivial. We now prove the sufficiency. Since $R_{\tilde{i}}\in R_{1,q-1}^2$, we have $\overline{N}_{\tilde{i}}\subseteq\overline{N}_{1,q-1}+\overline{N}_{1,q-1}$. Since $N_{1,q-1}$ generates $\mathbb{Z}_n$, we obtain $\varphi_i\pi(N_{\tilde{i}})\subseteq\langle\varphi_i\pi(N_{1,q-1})\rangle$ for $1\leq i\leq k$.
By Lemma \ref{pre-claim}, one gets $\pi(N_{\tilde{i}})\subseteq\pi(N_{1,q-1})+\pi(N_{1,q-1})$. Since $N_{1,q-1}$ is a union of $L$-cosets from Lemma \ref{isomorphic} \ref{isomorphic-1}, one has $N_{\tilde{i}}\subseteq N_{1,q-1}+N_{1,q-1}$. Thus, \ref{claim-1} is valid. The proof of \ref{claim-2} is similar, hence omitted.
\end{proof}

\begin{lemma}\label{R-1}
Let $(\alpha_0,\alpha_1,\dots,\alpha_l)$ be a path in the digraph ${\rm Cay}(\mathbb{Z}_n/L,\pi(N_{1,q-1}))$ with $\alpha_0=L$ and $l>0$. Suppose that $x_l\in\mathbb{Z}_n$ satisfies $\pi(x_l)=\alpha_l$. Then there exists a path $(y_0,y_1,\dots,y_l)$ in $\Delta_{q}$ such that
\begin{align}
y_0&=0,\nonumber\\
\pi(y_i)&=\alpha_i\quad(0<i<l),\nonumber\\
y_l&=x_l.\nonumber
\end{align}
\end{lemma}
\begin{proof}
By Lemma \ref{isomorphic} \ref{isomorphic-1}, $N_{1,q-1}$ is a union of $L$-cosets. Now, choosing $y_i\in\pi^{-1}(\alpha_i)$ arbitrarily for $1\leq i\leq l-1$, we find a path $(0,y_1,y_2,\ldots,y_l=x_{l})$ in $\Delta_{q}$ with the desired properties.
\end{proof}

\begin{lemma}\label{R}
Let $(L_0,L_1,\dots,L_l)$ be a path in ${\rm Cay}(U_0/L,\overline{N}_{1,q-1})$ with $L_0=L$ and $l>1$. Suppose that $x_l\in\mathbb{Z}_n$ satisfies $\varphi_0\pi(x_l)=L_l$. Then there exists a path $(y_0,y_1,\dots,y_l)$ in $\Delta_{q}$ such that
\begin{align}
y_0&=0,\nonumber\\
\varphi_0\pi(y_i)&=L_i\quad(0<i<l),\nonumber\\
y_l&=x_l.\nonumber
\end{align}
\end{lemma}
\begin{proof}
Note that $N_{1,q-1}$ generates $\mathbb{Z}_n$. In view of Lemma \ref{genrators} \ref{genrators-1}, there exists a path $(\alpha_0,\alpha_{1},\ldots,\alpha_l)$ in ${\rm Cay}(\mathbb{Z}_n/L,\pi(N_{1,q-1}))$ such that $\varphi_0(\alpha_{i})=L_i$ for $0\leq i\leq l$. By Lemma \ref{R-1}, there exists a path $(y_0=0,y_1,\dots,y_l=x_l)$ in $\Delta_{q}$ such that $\pi(y_i)=\alpha_{i}$ for $0\leq i\leq l$. The desired result follows.
\end{proof}

\begin{lemma}\label{nonsymmetric}
Let $l$ be a positive integer. If $\Gamma_{h,l}\in F_{q}$ with $h>2$, then $R_{h,l}\notin R_{1,q-1}^{i}$ for $1<i<h$.
\end{lemma}
\begin{proof}
Suppose, to the contrary that $R_{h,l}\in R_{1,q-1}^{i}$ for some $i\in\{2,3,\ldots,h-1\}$. Pick $x\in N_{h,l}$. Then $(L,\varphi_0\pi(x))\in R_{h,l}\in R_{1,q-1}^{i}$, so there exists a path of length $i$ from $L$ to $\varphi_0\pi(x)$ in ${\rm Cay}(U_0/L,\overline{N}_{1,q-1})$. By Lemma \ref{R}, we have $i\geq\partial_{\Delta_q}(0,x)\geq\partial_{\Gamma}(0,x)=h$, a contradiction.
\end{proof}

\begin{lemma}\label{h=2}
If $\overline{N}_{h,l}\subseteq\overline{N}_{1,q-1}+\overline{N}_{1,q-1}$ with $h>1$ and $l>0$, then $h=2$.
\end{lemma}
\begin{proof}
By Lemma \ref{claim} \ref{claim-1}, we have $\Gamma_{h,l}\in\Gamma_{1,q-1}^2\subseteq F_q$. Lemma \ref{nonsymmetric} implies $h=2$.
\end{proof}

\begin{lemma}\label{sym}
Assume $q=3$ and $R_{\tilde{i}}$ is symmetric with $R_{\tilde{i}}\neq R_{0,0}$. Suppose $L_0\in \overline{N}_{1,2}$ and $aL_0\in \overline{N}_{\tilde{i}}$ with $a\in\mathbb{Z}$. Then there exists an integer $b$ such that $abL_0=L$, $\overline{N}_{1,2}^{(b-1)}=\overline{N}_{1,2}$ and $bL_0\in \overline{N}_{\tilde{j}}$, where $R_{\tilde{j}}\in R_{1,2}^2\setminus\{R_{1,2},R_{2,1}\}$ and no element in $\overline{N}_{\tilde{j}}$ is a generator of $U_0/L$.
\end{lemma}
\begin{proof}
Since $\mathcal{A}_0$ is an orbit S-ring, there exists a subgroup $K\leq {\rm Aut}(U_0/L)$ such that every basic set of $\mathcal{A}_0$ is a $K$-orbit. In particular, $\overline{N}_{\tilde{i}}$ and $\overline{N}_{1,2}$ are $K$-orbits. Since $R_{\tilde{i}}$ is symmetric, we have $aL_0,-aL_0\in \overline{N}_{\tilde{i}}$. Thus, there exists $f\in K$ such that $f(aL_0)=-aL_0$ and $f(\overline{N}_{1,2})=\overline{N}_{1,2}$. Since $U_0/L$ is cyclic, there exists an integer $b$ such that $f(L')=(b-1)L'$ for $L'\in U_0/L$. Hence, $abL_0=L$ and $\overline{N}_{1,2}^{(b-1)}=\overline{N}_{1,2}$.

Since $R_{\tilde{i}}\neq R_{0,0}$, one has $|U_0/L|\nmid a$. The fact $abL_0=L$ implies $|U_0/L|\mid ab$. It follows that $|U_0/L|$ and $b$ are not coprime. Hence, $bL_0$ is not a generator of $U_0/L$. By Lemma \ref{genrators} \ref{genrators-3}, there exists $\overline{N}_{\tilde{j}}\in\mathcal{S}(\mathcal{A}_0)\setminus\{\overline{N}_{1,2},\overline{N}_{2,1}\}$ with $bL_0\in\overline{N}_{\tilde{j}}$. The fact $\mathcal{A}_0$ is an orbit S-ring implies that no element in $\overline{N}_{\tilde{j}}$ is a generator of $U_0/L$. Since $bL_0=L_0+(b-1)L_0$ and $\overline{N}_{1,2}^{(b-1)}=\overline{N}_{1,2}$, we have  $R_{\tilde{j}}\in R_{1,2}^2\setminus\{R_{1,2},R_{2,1}\}$.
\end{proof}

\begin{lemma}\label{sym-1}
Let $q=3$. The following hold:
\begin{enumerate}
\item\label{sym-1-1} if $R_{\tilde{i}}$ is symmetric with $R_{\tilde{i}}\neq R_{0,0}$, then no element in $\overline{N}_{\tilde{i}}$ is a generator of $U_0/L$;

\item\label{sym-1-2} if each element in $\overline{N}_{1,2}+\overline{N}_{1,2}$ is a generator of $U_0/L$, then $(U_0/L,\{R_{\tilde{i}}\}_{\Gamma_{\tilde{i}}\in F_3})$ is skew-symmetric.
\end{enumerate}
\end{lemma}
\begin{proof}
Pick $L_0\in\overline{N}_{1,2}$. By Lemma \ref{genrators} \ref{genrators-3}, $L_0$ is a generator of $U_0/L$.

\ref{sym-1-1} Note that there exists an integer $a$ such that $aL_0\in\overline{N}_{\tilde{i}}$. Since $\mathcal{A}_0$ is an orbit S-ring, it suffice to show that $aL_0$ is not a generator of $U_0/L$. By Lemma \ref{sym}, there exists an integer $b$ such that $abL_0=L$ and $\overline{N}_{1,2}^{(b-1)}=\overline{N}_{1,2}$. In view of Lemma \ref{genrators} \ref{genrators-5}, we have $-L_0\notin\overline{N}_{1,2}$. Since $(b-1)L_0\in\overline{N}_{1,2}$, one gets $|U_0/L|\nmid b$. It follows that $a$ and $|U_0/L|$ are not coprime. Then $aL_0$ is not a generator of $U_0/L$. Thus, \ref{sym-1-1} is valid.

\ref{sym-1-2} follows from Lemma \ref{sym}.
\end{proof}

\begin{prop}\label{(1,q-1)}
Assume $q>3$. Suppose $p_{(1,q-1),(1,q-1)}^{(2,q-2)}>0$ or $|\Gamma_{1,q-1}^{2}|=1$. If $m$ is the length of a shortest circuit in the subdigraph $\Delta_{q}$, then $\Delta_{q}$ is isomorphic to $C_m[\overline{K}_l]$, where $l=k_{1,q-1}$. Moreover, if $p_{(1,q-1),(1,q-1)}^{(2,q-2)}>0$, then $p_{(1,q-1),(1,q-1)}^{(2,q-2)}=k_{1,q-1}$ and $\Gamma_{1,q-1}^h=\{\Gamma_{h,q-h}\}$ for $1\leq h\leq q-1$.
\end{prop}
\begin{proof}
Let $(x_{0}=0,x_{1},\ldots,x_{m-1})$ be a shortest circuit in $\Delta_{q}$. Assume $(x_0,x_2)\in\Gamma_{a,b}$. It follows from Lemma \ref{jb4} that $k_{1,q-1}\geq k_{a,b}$. In the notation of Lemma~\ref{lem:2}, we get $(x_{i-1},x_{i+1})\in\Gamma_{a,b}$ and $P_{(q-1,1),(a,b)}(x_{i+2},x_{i+3})=Y_{i+1}$ for all $i$, where the indices are read modulo $m$. By Lemma~\ref{lem:2} applied to the circuit $(y_{i},x_{i+1},x_{i+2},x_{i+3},\ldots,x_{i-1})$ for any $y_{i}\in Y_{i}$, one has
$Y_{i}-x_{i-1}=P_{(1,q-1),(1,q-1)}(y_{i},x_{i+2})-y_{i}=P_{(q-1,1),(a,b)}(x_{i+2},x_{i+3})-y_{i}=Y_{i+1}-y_{i}.$ Since $x_j\in Y_j$ for $1\leq j<i$ and $x_0=0$, it follows by induction that $Y_{1}=Y_{i+1}-y_{i}$ for any $y_i\in Y_i$. Then $Y_1+Y_i=Y_{i+1}$ for $i>0$. For each nonpositive integer $i$, we have $i\equiv r~({\rm mod}~q)$ for some $r\in\{1,\ldots,q\}$, which implies $Y_1+Y_{i}=Y_1+Y_r=Y_{r+1}=Y_{i+1}$. Thus, for all integers $i$,
\begin{equation}\label{step2}
Y_1+Y_i=Y_{i+1}.
\end{equation}

\textbf{Case 1.} $m>q$.

In view of Lemma \ref{lem}, we obtain $p_{(1,q-1),(1,q-1)}^{(2,q-2)}=0$, and so $\Gamma_{1,q-1}^{2}=\{\Gamma_{a,b}\}$. By Lemma \ref{jb} \ref{jb-1}, one has $k_{1,q-1}^2=p_{(1,q-1),(1,q-1)}^{(a,b)}k_{a,b}$. Since $k_{1,q-1}\geq k_{a,b}$, we get $k_{a,b}=p_{(1,q-1),(1,q-1)}^{(a,b)}=k_{1,q-1}$, which implies $Y_i=\Gamma_{1,q-1}(x_{i-1})$ for all $i$. It follows from \eqref{step2} that $(y_{i},y_{i+1})\in\Gamma_{1,q-1}$ for any $y_i\in Y_i$ and $y_{i+1}\in Y_{i+1}$. In other words, all out-neighbors of vertices of $Y_i$ are contained in $Y_{i+1}$. This implies that $\bigcup_{i=1}^mY_i$ is a connected component of $\Delta_q$, and hence is equal to the set of vertices of $\Delta_q$. Thus, $\Delta_{q}$ is isomorphic to $C_m[\overline{K}_l]$, where $l=k_{1,q-1}$.

\textbf{Case 2.} $m=q$.

Note that $\partial_{\Gamma}(x_i,x_{i+j})=j$ with $1\leq j<q$ for all $i$. It follows that $p_{(1,q-1),(1,q-1)}^{(2,q-2)}>0$ and $(a,b)=(2,q-2)$. Suppose $k>0$ in \eqref{G/L}. For each $j$ with $1\leq j\leq k$, pick distinct $L_j,L'_j\in U_j/L\setminus\{L\}$. Let $a_{i}=(\varphi_0\pi(x_i),L_1,L_2,\ldots,L_k)$, $b_{i}=(\varphi_0\pi(x_i),L,L,\ldots,L)$ and $c_i=(\varphi_0\pi(x_i),L_1',L_2',\ldots,L_k')$ for $i\in\mathbb{Z}$. By Lemma \ref{genrators}, we have $h_{i+1}-h_i'\in\pi(N_{1,q-1})$ for distinct symbols $h,h'\in\{a,b,c\}$ with $0\leq i\leq q-1$. If $q$ is even, then $(b_0,a_1,h_2,a_3,b_4,a_5\ldots,b_{q-2},a_{q-1})$ is a circuit in ${\rm Cay}(\mathbb{Z}_n/L,\pi(N_{1,q-1}))$ for each $h\in\{b,c\}$; if $q$ is odd, then $(b_0,a_1,h_2,h'_3,a_4,h_5',a_6\ldots,h_{q-2}',a_{q-1})$ is a circuit in ${\rm Cay}(\mathbb{Z}_n/L,\pi(N_{1,q-1}))$ for each $(h,h')\in\{(b,c),(c,b)\}$. Pick $y,z\in\mathbb{Z}_n$ with $\pi(y)=b_2$ and $\pi(z)=c_2$. By Lemma \ref{R-1}, there exist two circuits $(y_0=0,y_1,y_2=y,\ldots,y_{q-1})$ and $(z_0=0,z_1,z_2=z,\ldots,z_{q-1})$ in $\Delta_q$. Hence, $y,z\in N_{2,q-2}$. It follows that $b_2,c_2\in\pi(N_{2,q-2})$, contrary to Lemma \ref{basic sets}. Thus, $k=0$ and $U_0=\mathbb{Z}_n$.

We will prove
\begin{equation}
Y_{1}^{(i)}\subseteq Y_{i}\label{l=q}
\end{equation}
for all integers $i$ by induction on $|i|$. The case $i=0$ is trivial. Suppose $Y_{1}^{(i)}\subseteq Y_{i}$. Since $Y_i\pm Y_1=Y_{i\pm1}$ from \eqref{step2}, we get $Y_{1}^{(i\pm1)}\subseteq Y_{1}^{(i)}\pm Y_1\subseteq Y_{i\pm1}$. Thus, \eqref{l=q} is valid.

We next claim that $q$ and $|\mathbb{Z}_n/L|$ are not coprime. By \eqref{l=q}, one has $\pi(Y_1)^{(q)}=\pi(Y_1^{(q)})\subseteq \pi(Y_q)=\pi(Y_0)$. Since $Y_1\subseteq N_{1,q-1}$, from Lemma \ref{genrators} \ref{genrators-3}, all elements of $\pi(Y_1)$ are generators of $\mathbb{Z}_n/L$. Suppose, to the contrary, $q$ and $|\mathbb{Z}_n/L|$ are coprime. Then all elements of $\pi(Y_1)^{(q)}$ are generators of $\mathbb{Z}_n/L$. Moreover, $|\pi(Y_1)^{(q)}|=|\pi(Y_1)|=|\pi(Y_0)|$ since $|Y_1|=|Y_0|$. Thus, $\pi(Y_1)^{(q)}=\pi(Y_0)$. This is a contradiction since $\pi(Y_0)$ contains the identity which is not a generator of $\mathbb{Z}_n/L$. Thus, our claim is valid.

By Lemma \ref{genrators} \ref{genrators-4}, we may assume that $x_1$ is a generator of $\mathbb{Z}_n$. The fact that $(y_{i},x_{i+1},x_{i+2},x_{i+3},\ldots,x_{i-1})$ is a circuit for any $y_{i}\in Y_{i}$ with $1\leq i\leq q-1$ implies $Y_{i}\subseteq N_{i,q-i}$. Let $x\in N_{1,q-1}$. Then $x=rx_1$ for some $r\in\mathbb{Z}$. Then $\langle r \pi(x_1)\rangle=\langle\pi(x)\rangle=\mathbb{Z}_n/L$ by Lemma \ref{genrators} \ref{genrators-3}. This implies $r$ and $|\mathbb{Z}_n/L|$ are coprime. By the
claim, we see that $r$ is not divisible by $q$. So $r\equiv r'~({\rm mod}~q)$ for some $0<r'<q$. Then $x=rx_1\in Y_1^{(r)}\subseteq Y_r=Y_{r'}\subseteq N_{r',q-r'}$ by \eqref{l=q}. This forces $r'=1$, and hence $x\in Y_1$.
Since $x\in N_{1,q-1}$ was arbitrary, we conclude $N_{1,q-1}\subseteq Y_1$. Since $N_{1,q-1}\supseteq Y_1$ clearly
holds, we obtain $p_{(1,q-1),(1,q-1)}^{(2,q-2)}=k_{1,q-1}$. By Lemma \ref{pure}, the first and second statement are both valid.
\end{proof}

\section{The case $\min T=3$}

For the purpose of proving Theorem \ref{Main2}, Proposition \ref{(1,q-1)} will turn out to be enough to treat the case $q>3$, where $T=\{q\}$. Indeed, $T=\{q\}$ implies $p_{(1,q-1),(1,q-1)}^{(2,q-2)}>0$. In order to deal with the case $q=3$, we prove the following proposition which determines the subdigraph $\Delta_{3}$ in a more general setting, namely, $\min T=3$. Then the girth of $\Gamma$ is $3$, and so
\begin{align}
\Gamma_{2,1}\in\Gamma_{1,2}^2.\label{girth 3}
\end{align}

\begin{prop}\label{q=3}
Suppose $\min T=3$. Then $\Delta_{3}$ is isomorphic to one of the following:
\begin{enumerate}
\item\label{q=3-1} the digraphs in Theorem {\rm\ref{Main2} \ref{Main2-2}--\ref{Main2-4}};

\item\label{q=3-2} $C_3[\overline{K}_{l}]$, where $l\geq1$;

\item\label{q=3-3} $(C_3\times K_{h})[\overline{K}_{l}]$, where $h>3$, $3\nmid h$ and $l>1$.
\end{enumerate}
\end{prop}

We continue to use the notations introduced after Lemma \ref{pure}, where we take $q=3$.

\subsection{The sets $R_{1,2}^2$ and $R_{1,2}R_{2,1}$}

In this subsection, we determine the possible relations contained in $R_{1,2}^2$ and $R_{1,2}R_{2,1}$.

\begin{lemma}\label{q=3 step1-2}
Let $N_{\tilde{i}}\in\mathcal{S}(\mathcal{A})$. Then $\pi(N_{\tilde{i}})^{(-1)}\subseteq \pi(N_{\tilde{i}})+\pi(N_{\tilde{i}})$. In particular, if $N_{1,3}\in\mathcal{S}(\mathcal{A})$, then the following hold:
\begin{enumerate}
\item\label{q=3 step1-2-1} $N_{1,3}\subsetneq\pi^{-1}(\pi(N_{1,3}))$;

\item\label{q=3 step1-2-2} $\Gamma_{1,2}\notin\Gamma_{1,3}^2$.
\end{enumerate}
\end{lemma}
\begin{proof}
Since $N_{1,2}$ generates $\mathbb{Z}_n$ and $N_{1,2}^{(-1)}\subseteq N_{1,2}+N_{1,2}$ from \eqref{girth 3}, the first statement is valid by Lemma \ref{q=3 step1}.

Now we suppose $N_{1,3}\in\mathcal{S}(\mathcal{A})$. Since $\Gamma_{3,1}\notin\Gamma_{1,3}^2$, we have $N_{1,3}^{(-1)}\cap (N_{1,3}+N_{1,3})=\emptyset$. \ref{q=3 step1-2-1} follows from the first statement.

By \ref{q=3 step1-2-1}, $N_{1,3}$ is not a union of $L$-cosets, which implies that ${\rm rad}(N_{1,3})\neq L$. It follows that $N_{1,3}$ does not generate $\mathbb{Z}_n$. Since $N_{1,2}$ generates $\mathbb{Z}_n$, \ref{q=3 step1-2-2}  is valid.
\end{proof}

\begin{lemma}\label{q=3 step2-2}
Suppose $\Gamma_{1,3}\in F_{3}$. Then $\Gamma_{2,2}\notin\Gamma_{1,2}\Gamma_{1,3}$.
\end{lemma}
\begin{proof}
Suppose $\Gamma_{2,2}\in\Gamma_{1,2}\Gamma_{1,3}$. Then there exist $x\in N_{1,2}$ and $y\in N_{1,3}$ such that $x+y\in N_{2,2}$. Since $\Gamma_{2,2}\in\Gamma_{1,3}\Gamma_{1,2}$, there exists $z\in N_{1,3}$ such that $x+y+z\in N_{2,1}$. Since $(0,x,x+y,x+y+z)$ is a circuit of length $4$ consisting of arcs of type $(1,2),(1,3),(1,3)$ and $(1,2)$, we have $\partial_{\Gamma}(x+y+z,x)=2$ and $\partial_{\Gamma}(x,x+y+z)\leq2$. Then $y+z\in N_{1,2}\cup N_{2,2}$. Since $y,z\in N_{1,3}$ and $\Gamma_{1,2}\notin\Gamma_{1,3}^2$ from Lemma \ref{q=3 step1-2} \ref{q=3 step1-2-2}, one has $y+z\notin N_{1,2}$, which implies $y+z\in N_{2,2}$. It follows that $p_{(1,3),(1,3)}^{(2,2)}>0$. Since $4\in T$, we can use the second part of Proposition \ref{(1,q-1)} to obtain $p_{(1,3),(1,3)}^{(2,2)}=k_{1,3}$. This implies that $x\in\Gamma_{3,1}(x+y)=P_{(1,3),(1,3)}(0,x+y)$, contrary to the fact that $x\in N_{1,2}$.
\end{proof}

\begin{lemma}\label{q=3 step4}
The following hold:
\begin{enumerate}
\item\label{q=3 step4-1} $\{\Gamma_{2,1}\}\subseteq\Gamma_{1,2}^{2}\subseteq\{\Gamma_{1,2},\Gamma_{2,1},\Gamma_{2,3},\Gamma_{2,4}\}$;

\item\label{q=3 step4-2} $\{R_{2,1}\}\subseteq R_{1,2}^{2}\subseteq\{R_{1,2},R_{2,1},R_{2,3},R_{2,4}\}$.
\end{enumerate}
\end{lemma}
\begin{proof}
\ref{q=3 step4-1} By \eqref{girth 3}, we have $\{\Gamma_{2,1}\}\subseteq\Gamma_{1,2}^2\subseteq\{\Gamma_{1,2},\Gamma_{1,3},\Gamma_{1,4},\Gamma_{2,1},\Gamma_{2,2},\Gamma_{2,3},\Gamma_{2,4}\}$. Assume by way of contradiction, $\Gamma_{s,t}\in\Gamma_{1,2}^2$ for some $(s,t)\in\{(1,3),(1,4),(2,2)\}$.

Suppose $s=1$. Then $t\in\{3,4\}$ and $\Gamma_{1,t}\in F_{3}$. By Lemma \ref{q=3 step1-2}, there exist $x,y\in N_{1,t}$ such that $-x-y+L\in\pi(N_{1,t})\subseteq \pi(N_{1,2}+N_{1,2})$. Then there exists $z\in\mathbb{Z}_n$ such that $z+L,-x-y-z+L\in\pi(N_{1,2})$. By Lemma \ref{isomorphic} \ref{isomorphic-1}, $N_{1,2}$ is a union of $L$-cosets, which implies $z,-x-y-z\in N_{1,2}$. The fact that $(0,x,x+y,x+y+z)$ is a circuit in $\Gamma$ consisting of arcs of type $(1,t),(1,t),(1,2)$ and $(1,2)$ implies $t=3$, and so $(x,x+y+z)\in\Gamma_{2,2}$. Since $x+z\in P_{(1,2),(1,3)}(x,x+y+z)$, we have $\Gamma_{2,2}\in\Gamma_{1,2}\Gamma_{1,3}$, contrary to Lemma \ref{q=3 step2-2}. Thus, $(s,t)=(2,2)$. By Lemma \ref{claim} \ref{claim-1} with $q=3$, we obtain
\begin{align}
\{R_{2,1},R_{2,2}\}\subseteq R_{1,2}^{2}\subseteq\{R_{1,2},R_{2,1},R_{2,2},R_{2,3},R_{2,4}\}.\label{R^2}
\end{align}

By Lemma \ref{genrators} \ref{genrators-5}, $R_{1,2}$ is non-symmetric. It follows that $R_{0,0}\notin R_{1,2}^2$, and so $R_{2,2}\neq R_{0,0}$ from \eqref{R^2}. Then there exists a path $(L,L',L'')$ in ${\rm Cay}(U_0/L,\overline{N}_{1,2})$ such that $L''\in\overline{N}_{2,2}$. By Lemma \ref{genrators} \ref{genrators-3}, there exists an integer $a$ with $L''=aL'$. Since $R_{2,2}$ is symmetric, from Lemma \ref{sym}, there exists an integer $b$ such that $abL'=L$, $\overline{N}_{1,2}^{(b-1)}=\overline{N}_{1,2}$ and $bL'\in\overline{N}_{2,h}$ for some $h>1$ by \eqref{R^2}. In particular, since $L',(a-1)L'\in \overline{N}_{1,2}$, one gets $(b-1)L',(-a-b+1)L'=(b-1)(a-1)L'\in \overline{N}_{1,2}$. The fact $-bL'=(-a-b+1)L'+(a-1)L'$ implies $-bL'\in\overline{N}_{h,2}\subseteq\overline{N}_{1,2}+\overline{N}_{1,2}$. By Lemma \ref{h=2}, one gets $h=2$. Since $\mathcal{A}_0$ is an orbit S-ring, we obtain $bL'=caL'$ for some integer $c$, and so $b^2L'=abcL'=L$.

Since $\overline{N}_{1,2}^{(b-1)}=\overline{N}_{1,2}$, we have $(-2jb+1)L'=(b-1)^{2j}L'\in \overline{N}_{1,2}$ and $(2jb-b-1)L'=(b-1)^{2j-1}L'\in\overline{N}_{1,2}$ for $j\in\mathbb{Z}$. By Lemma \ref{sym-1} \ref{sym-1-1}, $bL'\in\overline{N}_{2,2}$ is not a generator of $U_0/L$. Then $(2j-1)bL'\in\overline{N}_{2,2}$ is also not a generator of $U_0/L$. Since $(2j-1)bL'=L'+(2jb-b-1)L'\in\overline{N}_{1,2}+\overline{N}_{1,2}$, from Lemma \ref{genrators} \ref{genrators-3} and \eqref{R^2}, one gets $(2j-1)bL'\in\overline{N}_{2,h}$ with $h>1$. Since $-(2j-1)bL'=(-2jb+1)L'+(b-1)L'$, we obtain $-(2j-1)bL'\in\overline{N}_{h,2}\subseteq\overline{N}_{1,2}+\overline{N}_{1,2}$. By Lemma \ref{h=2}, we obtain $h=2$, and so $(2j-1)bL'\in\overline{N}_{2,2}$. Since $j$ was arbitrary, $|U_0/L|$ is even.

Since $R_{2,1}\in R_{1,2}^2$ from \eqref{R^2}, there exist $L_1,L_2\in\overline{N}_{1,2}$ such that $-L'=L_1+L_2$. Since $|U_0/L|$ is even, from Lemma \ref{genrators} \ref{genrators-3}, there exist odd integers $u,v$ such that $L_1=uL'$ and $L_2=vL'$. It follows that $(1+u+v)L'=L$. Since $1+u+v$ is odd, this contradicts the fact that $L'$ has even order.

\ref{q=3 step4-2} is immediate from \ref{q=3 step4-1} and Lemma \ref{claim} \ref{claim-1}.
\end{proof}

\begin{lemma}\label{(2,2)}
Suppose $(2,2)\in\tilde{\partial}(\Gamma)$. Then $\Gamma_{2,2}\in\Gamma_{1,2}\Gamma_{1,3}\cup\Gamma_{1,3}^2$. Moreover, if $\Gamma_{1,3}\in F_3$, then $\Gamma_{2,2}\in\Gamma_{1,3}^2$.
\end{lemma}
\begin{proof}
Since $\min T=3$, one gets $\Gamma_{2,2}\in\Gamma_{1,2}^2\cup\Gamma_{1,2}\Gamma_{1,3}\cup\Gamma_{1,3}^2$. By Lemma \ref{q=3 step4} \ref{q=3 step4-1}, the first statement is valid. The second statement then follows from Lemma \ref{q=3 step2-2}.
\end{proof}

\begin{lemma}\label{q=3 step6}
If $\Gamma_{2,3}\in\Gamma_{1,2}^2$, then $\Gamma_{3,2}\in\Gamma_{1,2}^3$.
\end{lemma}
\begin{proof} Since $\Gamma_{2,3}\in\Gamma_{1,2}^2$, there exists a path $(x,y,z)$ consisting of arcs of type $(1,2)$ such that $(x,z)\in\Gamma_{2,3}$. Assume the contrary, namely, $\Gamma_{3,2}\notin\Gamma_{1,2}^3$. For any path $(z,w,u,x)$ in $\Gamma$, we have $\{(z,w),(w,u),(u,x)\}\nsubseteq\Gamma_{1,2}$.

We claim that $(z,u)\in\Gamma_{2,2}$ for each path $(z,w,u,x)$ with $(z,w)\notin\Gamma_{1,2}$. Note that $(z,w)\in\Gamma_{1,3}\cup\Gamma_{1,4}$. It follows that $(z,u)\in\Gamma_{2,2}\cup\Gamma_{2,3}$. Suppose $(z,u)\in\Gamma_{2,3}$. Since $\Gamma_{2,3}\in\Gamma_{1,2}^2$, there exists a vertex $w'\in P_{(1,2),(1,2)}(z,u)$. The fact $(z,x)\in\Gamma_{3,2}\notin\Gamma_{1,2}^3$ implies $P_{(1,2),(1,2)}(w',x)=\emptyset$, and so $(u,x)\notin\Gamma_{1,2}$. Since $\Gamma_{2,3}\in\Gamma_{1,2}^2$ again, we have $(w',x)\notin\Gamma_{2,3}$, and so $(w',x)\in\Gamma_{2,2}$. It follows that $(u,x)\in\Gamma_{1,3}$. Since $(x,y,z,w',u)$ is a path consisting of arcs of type $(1,2)$, we get $\Gamma_{1,3}\in F_3$. By Lemma \ref{q=3 step2-2}, one obtains $\Gamma_{2,2}\notin\Gamma_{1,2}\Gamma_{1,3}$, contrary to the fact that $u\in P_{(1,2),(1,3)}(w',x)$.

Let $(z,w,u,x)$ be a path in $\Gamma$. Since $\{(z,w),(w,u),(u,x)\}\nsubseteq\Gamma_{1,2}$, from Lemma \ref{commutativity}, we may assume $(z,w)\in\Gamma_{1,3}\cup\Gamma_{1,4}$. The claim implies $(z,u)\in\Gamma_{2,2}$, and so $(z,w)\in\Gamma_{1,3}$. In particular, $4\in T$. Suppose $(w,u),(u,x)\in\Gamma_{1,2}$. Since $(w,u,x,y,z)$ is a path consisting of arcs of type $(1,2)$, one has $\Gamma_{1,3}\in F_3$. By Lemma \ref{q=3 step2-2}, we get $\Gamma_{2,2}\notin\Gamma_{1,3}\Gamma_{1,2}$, contrary to the fact $w\in P_{(1,3),(1,2)}(z,u)$. Hence, $(w,u)$ or $(u,x)\notin\Gamma_{1,2}$. By Lemma \ref{commutativity}, there exists a path $(w,u',x)$ such that $(w,u')\notin\Gamma_{1,2}$. The claim implies $(z,u')\in\Gamma_{2,2}$, and so $(w,u')\in\Gamma_{1,3}$. Since $w\in P_{(1,3),(2,r)}(z,x)$ and $u'\in P_{(1,3),(1,s)}(w,x)$ for some $r,s>1$, from the commutativity of $\Gamma$, there exist $u''\in P_{(2,r),(1,3)}(z,x)$ and $w''\in P_{(1,3),(1,s)}(z,u'')$. By the claim, we get $r=2$. Since $w\in P_{(1,3),(1,3)}(z,u')$, one obtains $p_{(1,3),(1,3)}^{(2,2)}>0$, which implies that there exists $v\in P_{(1,3),(1,3)}(w,x)$. Since $4\in T$, from Proposition \ref{(1,q-1)}, one has $\Gamma_{1,3}^3=\{\Gamma_{3,1}\}$. Since $(z,w,v,x)$ is a path consisting of arcs of type $(1,3)$, one gets $(x,z)\in\Gamma_{1,3}$, a contradiction.
\end{proof}

\begin{lemma}\label{q=3 step5}
We have $R_{1,2}R_{2,1}\subseteq\{R_{0,0},R_{1,2},R_{2,1},R_{1,3},R_{3,1},R_{2,3},R_{3,2},R_{3,3}\}$. Moreover, if $R_{1,3}\in R_{1,2}R_{2,1}\setminus\{R_{0,0},R_{1,2},R_{2,1},R_{2,3},R_{3,2},R_{3,3}\}$, then $R_{1,3}$ is symmetric.
\end{lemma}
\begin{proof}We claim that if $\Gamma_{2,2}\in\Gamma_{1,3}^2$, then $p_{(1,3),(1,3)}^{(2,2)}=k_{1,3}$, $\Gamma_{1,3}^2=\{\Gamma_{2,2}\}$ and $\Gamma_{1,3}^3=\{\Gamma_{3,1}\}$. Since $p_{(1,3),(1,3)}^{(2,2)}>0$ and $4\in T$, we can use the second part of Proposition \ref{(1,q-1)} to obtain our claim.

Since
\begin{align}
\Gamma_{1,2}\Gamma_{2,1}\subseteq\{\Gamma_{0,0},\Gamma_{1,2},\Gamma_{2,1},\Gamma_{1,3},\Gamma_{3,1},\Gamma_{2,2},\Gamma_{2,3},\Gamma_{3,2},\Gamma_{3,3}\},\label{1,2}
\end{align}
Lemma \ref{claim} \ref{claim-2} implies that the first statement follows if we show $R_{2,2}\notin R_{1,2}R_{2,1}\setminus\{R_{0,0},R_{1,2},R_{2,1},R_{2,3},R_{3,2},R_{3,3}\}$. In order to prove the second statement at the same time, we suppose
\begin{align}
R_{i,j}\in R_{1,2}R_{2,1}\setminus\{R_{0,0},R_{1,2},R_{2,1},R_{2,3},R_{3,2},R_{3,3}\}\label{R_ij}
\end{align}
with $(i,j)\in\{(1,3),(2,2)\}$. Then $\Gamma_{i,j}\in F_3$. Lemma \ref{claim} \ref{claim-2} implies
\begin{align}
\Gamma_{i,j}\in\Gamma_{1,2}\Gamma_{2,1}.\label{i,j}
\end{align}
By Lemma \ref{q=3 step1-2}, we have
\begin{align}
\pi(N_{i,j})^{(-1)}\subseteq\pi(N_{i,j})+\pi(N_{i,j}).\label{N_ij}
\end{align}

Suppose $N_{i,j}=\pi^{-1}(\pi(N_{i,j}))$. Lemma \ref{q=3 step1-2} \ref{q=3 step1-2-1} implies $(i,j)=(2,2)$. By Lemma \ref{(2,2)}, one gets $\Gamma_{2,2}\in\Gamma_{1,2}\Gamma_{1,3}\cup\Gamma_{1,3}^2$. In view of \eqref{i,j}, we obtain $\Gamma_{2,2}\in\Gamma_{1,2}\Gamma_{2,1}$. If $\Gamma_{2,2}\in\Gamma_{1,2}\Gamma_{1,3}$, then $\Gamma_{1,3}\in\Gamma_{1,2}\Gamma_{2,1}^2$, which implies $\Gamma_{1,3}\in F_{3}$, contrary to Lemma \ref{q=3 step2-2}. Thus, $\Gamma_{2,2}\in\Gamma_{1,3}^2$. By \eqref{N_ij}, one has $N_{2,2}^{(-1)}\subseteq N_{2,2}+N_{2,2}$. It follows that there exists a sequence $(x_0=0,x_2,x_4)$ such that $(x_{0},x_{2}),(x_2,x_4),(x_4,x_0)\in\Gamma_{2,2}$. Since $p_{(1,3),(1,3)}^{(2,2)}>0$, there exist vertices $x_1,x_3,x_5$ such that $(x_0,x_1,\ldots,x_5)$ is a circuit of length $6$ consisting of arcs of type $(1,3)$.
Since $\Gamma_{1,3}^3=\{\Gamma_{3,1}\}$ from the claim, we have $(x_0,x_{3}),(x_3,x_0)\in\Gamma_{3,1}$, a contradiction. Thus, $N_{i,j}\subsetneq\pi^{-1}(\pi(N_{i,j}))$.

Let $X=\{(i',j')\mid\Gamma_{i',j'}\in F_3,~\pi(N_{i',j'})=\pi(N_{i,j})\}$. Then $\{(i,j)\}\subsetneq X$ and
\begin{align}
\pi^{-1}(\pi(N_{i,j}))=\bigcup_{(i',j')\in X}N_{i',j'}.\label{X-1}
\end{align}
Let $(i',j')$ be an element in $X$. Since $R_{i',j'}=R_{i,j}$, it follows from \eqref{R_ij} that $(i',j')\notin\{(0,0),(1,2),(2,1),(2,3),(3,2),(3,3)\}$. Since $R_{i',j'}=R_{i,j}\in R_{1,2}R_{2,1}$, Lemma \ref{claim} \ref{claim-2} implies that $\Gamma_{i',j'}\in\Gamma_{1,2}\Gamma_{2,1}$. By \eqref{1,2}, we have $(i',j')\in\{(1,3),(3,1),(2,2)\}$. If $(2,2)\in X$, then $\pi(N_{i,j})=-\pi(N_{i,j})$, which implies $X=\{(1,3),(3,1),(2,2)\}$ since $\{(i,j)\}\subsetneq X$; if $(2,2)\notin X$, then $(i,j)=(1,3)$, and $X=\{(1,3),(3,1)\}$. It follows that
\begin{align}
X=\{(1,3),(3,1)\}~{\rm or}~\{(1,3),(3,1),(2,2)\}.\label{X}
\end{align}
In both cases, $\pi(N_{i,j})=-\pi(N_{i,j})$. Setting $(i,j)=(1,3)$, the second statement follows. Furthermore, $\eqref{X}$ implies $\Gamma_{1,3}\in F_3$.

We now suppose $(i,j)=(2,2)$. By \eqref{N_ij}, there exists a sequence $(0,y_2,y_4)$ such that $y_2,y_4-y_2\in N_{2,2}$ and $-\pi(y_4)\in\pi(N_{2,2})$. Since $\Gamma_{1,3}\in F_3$, from Lemma \ref{(2,2)}, one has $\Gamma_{2,2}\in\Gamma_{1,3}^2$, and so $p_{(1,3),(1,3)}^{(2,2)}>0$. It follows that there exist vertices $y_1,y_3$ such that $(0,y_{1},y_2,y_3,y_4)$ is a path consisting of arcs of type $(1,3)$. By the claim, we get $\Gamma_{1,3}^3=\{\Gamma_{3,1}\}$, which implies $(0,y_3)\in\Gamma_{3,1}$. Since $-\pi(y_4)\in\pi(N_{2,2})$, from \eqref{X-1} and \eqref{X}, we obtain $y_4\in N_{1,3}\cup N_{3,1}\cup N_{2,2}$. If $y_4\in N_{1,3}$ or $N_{3,1}$, then $(y_3,0,y_4)$ or $(y_3,y_4,0)$ respectively, is a path consisting of arcs of type $(1,3)$, which implies $y_4$ or $0\in\Gamma_{2,2}(y_3)$ since $\Gamma_{1,3}^2=\{\Gamma_{2,2}\}$ from the claim, a contradiction. If $y_4\in N_{2,2}$, then $y_3\in\Gamma_{3,1}(y_4)=P_{(1,3),(1,3)}(0,y_4)$ since $p_{(1,3),(1,3)}^{(2,2)}=k_{1,3}$ from the claim, a contradiction. The first statement is valid.
\end{proof}

\subsection{The relations $R_{2,3}$ and $R_{2,4}$}

In this subsection, we give some consequences under the assumption that $R_{2,3}$ or $R_{2,4}\in R_{1,2}^2$, which will be used in the proof of Proposition \ref{q=3}.

\begin{lemma}\label{3,3}
Let $s\in\{3,4\}$ and $(y_0,y_1,y_2,y_3,y_4,y_5)$ be a circuit consisting of arcs of type $(1,2)$ such that $\{(y_0,y_2),(y_2,y_4),(y_4,y_0)\}\subseteq\Gamma_{2,s}$. Suppose that $s=4$ or no element in $\overline{N}_{2,3}$ is a generator of $U_0/L$. Then $(y_0,y_3)\in\Gamma_{3,3}$.
\end{lemma}
\begin{proof}
Since $2\leq\partial_{\Gamma}(y_0,y_4)-1\leq\partial_{\Gamma}(y_0,y_3)\leq 3$ and $2\leq\partial_{\Gamma}(y_2,y_0)-1\leq\partial_{\Gamma}(y_3,y_0)\leq3$, we have $(y_0,y_3)\in\Gamma_{2,2}\cup\Gamma_{2,3}\cup\Gamma_{3,2}\cup\Gamma_{3,3}$.

Assume the contrary, namely, $(y_0,y_3)\in\Gamma_{2,2}\cup\Gamma_{2,3}\cup\Gamma_{3,2}$. Then $(y_4,y_0)\in\Gamma_{2,3}$ or $(y_0,y_2)\in\Gamma_{2,3}$, and so $s=3$. By the assumption, no element in $\overline{N}_{2,3}$ is a generator of $U_0/L$. Since $\mathcal{A}_0$ is an orbit S-ring, $\overline{N}_{2,3}$ does not generate $U_0/L$.

Suppose $(y_0,y_3)\in\Gamma_{2,3}\cup\Gamma_{3,2}$. It follows that $\varphi_0\pi(y_3-y_0)\in\overline{N}_{2,3}^{(\pm1)}$. Since $(y_4,y_0)\in\Gamma_{2,3}$, one has $\varphi_0\pi(y_0-y_4)\in\overline{N}_{2,3}$, which implies $\varphi_0\pi(y_3-y_4)\in\overline{N}_{2,3}^{(\pm1)}+\overline{N}_{2,3}$. Since $\overline{N}_{2,3}$ does not generate $U_0/L$, $\varphi_0\pi(y_3-y_4)\in\overline{N}_{2,1}$ is not a generator, contrary to Lemma \ref{genrators} \ref{genrators-3}.

Suppose $(y_0,y_3)\in\Gamma_{2,2}$. By Lemma \ref{(2,2)}, we have $\Gamma_{2,2}\in\Gamma_{1,2}\Gamma_{1,3}\cup\Gamma_{1,3}^2$. Suppose $\Gamma_{2,2}\in\Gamma_{1,2}\Gamma_{1,3}$. Then there exists $y\in P_{(1,3),(1,2)}(y_0,y_3)$. Since $(y,y_3,y_4,y_5,y_0)$ is a path consisting of arcs of type $(1,2)$ and $(y_0,y)\in\Gamma_{1,3}$, we have $\Gamma_{1,3}\in F_3$, contrary to Lemma \ref{q=3 step2-2}. Hence, $\Gamma_{2,2}\in\Gamma_{1,3}^2$, and so $p_{(1,3),(1,3)}^{(2,2)}>0$. Then there exists $y_4'\in P_{(1,3),(1,3)}(y_3,y_0)$. Since $4\in T$, from Proposition \ref{(1,q-1)}, one gets $p_{(1,3),(1,3)}^{(2,2)}=k_{1,3}$ and $\Gamma_{1,3}^{3}=\{\Gamma_{3,1}\}$. Since $(y_2,y_3,y_4',y_0)$ is a shortest path from $y_2$ to $y_0$ in $\Gamma$, we obtain $\partial_{\Gamma}(y_2,y_4')=2$. Since $y_3\in P_{(1,2),(1,3)}(y_2,y_4')$ and $p_{(1,3),(1,3)}^{(2,2)}=k_{1,3}$, one has $\partial_{\Gamma}(y_4',y_2)\neq2$. The fact that $(y_4',y_0,y_1,y_2)$ is a path implies $(y_2,y_4')\in\Gamma_{2,3}$, and so $p_{(1,2),(1,3)}^{(2,3)}>0$. Since $(y_0,y_2)\in\Gamma_{2,3}$, there exists $y_1'\in P_{(1,3),(1,2)}(y_0,y_2)$. Since $(y_3,y_4',y_0,y_1')$ is a path consisting of arcs of type $(1,3)$ and $\Gamma_{1,3}^3=\{\Gamma_{3,1}\}$, we get $(y_3,y_1')\in\Gamma_{3,1}$. The fact $y_2\in P_{(1,2),(1,2)}(y_1',y_3)$ implies $\Gamma_{1,3}\in\Gamma_{1,2}^2$, contrary to Lemma \ref{q=3 step4} \ref{q=3 step4-1}.
\end{proof}

\begin{lemma}\label{R2s}
Let $s\in\{3,4\}$. If $R_{2,s}\in R_{1,2}^2\setminus\{R_{1,2},R_{2,1}\}$, then $\{(i,j)\mid R_{i,j}=R_{2,s}\}=\{(2,s)\}$.
\end{lemma}
\begin{proof}
Let $X=\{\Gamma_{i,j}\mid R_{i,j}=R_{2,s}\}$. By Lemma \ref{claim} \ref{claim-1}, we have $X\subseteq\Gamma_{1,2}^2\setminus\{\Gamma_{1,2},\Gamma_{2,1}\}$. By Lemma \ref{q=3 step4} \ref{q=3 step4-1}, one gets $\{\Gamma_{2,s}\}\subseteq X\subseteq\{\Gamma_{2,3},\Gamma_{2,4}\}$. If $X=\{\Gamma_{2,3},\Gamma_{2,4}\}$, from Lemma \ref{q=3 step6}, then $R_{4,2}=R_{3,2}\in R_{1,2}^3$, contrary to Lemma \ref{nonsymmetric}. Thus, $X=\{\Gamma_{2,s}\}$.
\end{proof}

\begin{lemma}\label{s=4}
Let $s\in\{3,4\}$. If $R_{2,s}\in R_{1,2}^2\setminus\{R_{1,2},R_{2,1}\}$ and $|\overline{N}_{2,s}|=1$, then $s=4$.
\end{lemma}
\begin{proof}
Suppose $s=3$. By the assumption and Lemma \ref{claim} \ref{claim-1}, one gets $\Gamma_{2,3}\in\Gamma_{1,2}^2$. Lemma \ref{q=3 step6} implies $R_{3,2}\in R_{1,2}^3$, and so $R_{2,3}R_{1,2}\cap R_{2,1}^2\neq\emptyset$. Since $|\overline{N}_{2,3}|=1$, we have $|R_{2,3}R_{1,2}|=1$, and so  $R_{2,3}R_{1,2}\subseteq R_{2,1}^2$. In view of Lemma \ref{q=3 step4} \ref{q=3 step4-2}, we obtain $R_{2,3}R_{1,2}\subseteq\{R_{1,2},R_{2,1},R_{3,2},R_{4,2}\}$. If $R_{2,3}R_{1,2}=\{R_{2,1}\}$ or $\{R_{4,2}\}$, then $R_{3,2}\in R_{1,2}^2$ or $R_{4,2}\in R_{2,3}R_{1,2}\subseteq R_{1,2}^3$, contrary to Lemma \ref{nonsymmetric}. By Lemma \ref{q=3 step4} \ref{q=3 step4-2} again, one gets $R_{2,1}\in R_{1,2}^2$. Since $R_{2,s}\in R_{1,2}^2\setminus\{R_{2,1}\}$, one has $|\overline{N}_{1,2}|>1=|\overline{N}_{3,2}|^2$. It follows from Lemma \ref{jb} \ref{jb-1} that $R_{1,2}\notin R_{3,2}^2$, and so $R_{3,2}\notin R_{2,3}R_{1,2}$. Then $R_{2,3}R_{1,2}=\{R_{1,2}\}$. Therefore, $R_{2,3}\in R_{1,2}R_{2,1}\cap R_{1,2}^2$. Hence, $\{R_{1,2},R_{2,1}\}\subseteq R_{2,3}R_{2,1}$, contrary to the fact that $|\overline{N}_{2,3}|=1$. Thus, $s=4$.
\end{proof}

\begin{lemma}\label{q=3 step7}
Suppose that $R_{2,s}\in R_{1,2}^2\setminus\{R_{1,2},R_{2,1}\}$ for some $s\in\{3,4\}$. Assume $s=4$ or no element in $\overline{N}_{2,3}$ is a generator of $U_0/L$. Then the following hold:
\begin{enumerate}
\item\label{q=3 step7-1} $R_{3,3}\in R_{1,2}R_{2,1}\setminus\{R_{0,0}\}$;

\item\label{q=3 step7-2} $|U_0/L|=3h$ for some positive integer $h$;

\item\label{q=3 step7-3} $R_{2,4}\in R_{1,2}^2\setminus\{R_{1,2},R_{2,1}\}$, $R_{2,4}R_{1,2}=\{R_{3,3}\}$ and $R_{4,2}\in R_{3,3}R_{1,2}$.
\end{enumerate}
Moreover, for $L_0\in\overline{N}_{1,2}$, the following hold:
\begin{enumerate}[start=4]
\item\label{q=3 step7-4} $3\mid q$ and $(q/3,h)=1$ whenever $q$ is an integer with $qL_0\in\overline{N}_{3,3}$;

\item\label{q=3 step7-5} $\overline{N}_{2,4}=\{uL_0\}$ for some $u\in\{h,2h\}$.
\end{enumerate}
\end{lemma}
\begin{proof}
We may assume $s=4$ if $R_{2,4}\in R_{1,2}^2\setminus\{R_{1,2},R_{2,1}\}$, and $s=3$ otherwise. 
Note that $\mathcal{A}_0$ is an orbit S-ring and $N_{1,2}^{(-1)}\subseteq N_{1,2}+N_{1,2}$ from Lemma \ref{q=3 step4} \ref{q=3 step4-1}. By Lemma \ref{q=3 step1-1}, there exists an integer $a$ such that
\begin{align}
\overline{N}_{\tilde{i}}^{(a)}=\overline{N}_{\tilde{i}}^{(-a-1)}=\overline{N}_{\tilde{i}}\label{a}
\end{align}
for all $\Gamma_{\tilde{i}}\in F_3$. In view of Lemma \ref{genrators} \ref{genrators-3}, each element in $\overline{N}_{1,2}$ is a generator of $U_0/L$. In particular, $L_0$ is a generator of  $U_0/L$. Since $\overline{N}_{2,s}\subseteq\overline{N}_{1,2}+\overline{N}_{1,2}$ by the assumption, there exist $L_1\in\overline{N}_{2,s}$ and $L_2\in\overline{N}_{1,2}$ such that $L_1=L_0+L_2$. It follows that there exists an integer $b$ such that $L_1=bL_0$. Then $(b-1)L_0\in\overline{N}_{1,2}$.

Since $(b-1)L_0\in\overline{N}_{1,2}$, from \eqref{a}, $(L,L_0,bL_0,(a+b)L_0,(a+1)bL_0,(a+1)(b-1)L_0)$ is a circuit in ${\rm Cay}(U_0/L,\overline{N}_{1,2})$ such that $(bL_0,(a+1)bL_0),((a+1)bL_0,L)\in R_{2,s}$. By Lemma \ref{R}, there exists a circuit $(y_0=0,y_1,y_2,y_3,y_4,y_5)$ in $\Delta_{3}$ such that $\varphi_0\pi(y_1)=L_0$, $\varphi_0\pi(y_2)=bL_0$, $\varphi_0\pi(y_3)=(a+b)L_0$, $\varphi_0\pi(y_4)=(a+1)bL_0$ and $\varphi_0\pi(y_5)=(a+1)(b-1)L_0$. Since $\{(L,\varphi_0\pi(y_2)),(\varphi_0\pi(y_2),\varphi_0\pi(y_4)),(\varphi_0\pi(y_4),L)\}\subseteq R_{2,s}$, from the assumption and Lemma \ref{R2s}, one gets $\{(0,y_2),(y_2,y_4),(y_4,0)\}\subseteq\Gamma_{2,s}$.

Since $(b-1)L_0\in\overline{N}_{1,2}$, from \eqref{a}, $(L,(b-1)L_0,\varphi_0\pi(y_2))$ and $(\varphi_0\pi(y_4),(a+1)L_0,L)$ are paths in ${\rm Cay}(U_0/L,\overline{N}_{1,2})$. By Lemma \ref{R}, there exist vertices $y_1'\in P_{(1,2),(1,2)}(0,y_2)$ and $y_5'\in P_{(1,2),(1,2)}(y_4,0)$ such that $\varphi_0\pi(y_1')=(b-1)L_0$ and $\varphi_0\pi(y_5')=(a+1)L_0$. Lemma \ref{3,3} implies $\{(y_0,y_{3}),(y_1,y_4),(y_1',y_4),(y_2,y_5')\}\subseteq\Gamma_{3,3}$. Then $\{\varphi_0\pi(y_3),\varphi_0\pi(y_1-y_4),\varphi_0\pi(y_1'-y_4),\varphi_0\pi(y_2-y_5')\}\subseteq\overline{N}_{3,3}$.

By \eqref{a}, one has $\varphi_0\pi(y_3-y_1')=(a+1)L_0\in\overline{N}_{2,1}$. Since $\varphi_0\pi(y_1')+\varphi_0\pi(y_3-y_1')=\varphi_0\pi(y_3)=(a+b)L_0\in\overline{N}_{3,3}$, we get $\overline{N}_{3,3}\subseteq\overline{N}_{1,2}+\overline{N}_{2,1}$. In view of the assumption, one obtains $\overline{N}_{2,s}\cap(\overline{N}_{1,2}\cup\overline{N}_{2,1})=\emptyset$. Since $bL_0\in\overline{N}_{2,s}$ and $-aL_0\in-\overline{N}_{1,2}^{(a)}=-\overline{N}_{1,2}$ from \eqref{a}, we have $bL_0\neq-aL_0$, and so $\overline{N}_{3,3}\neq\overline{N}_{0,0}$. Thus, \ref{q=3 step7-1} is valid.

Let $c=a+b$. Since $cL_0=\varphi_0\pi(y_3)\in\overline{N}_{3,3}$, $(a^2-1-ac)L_0=\varphi_0\pi(y_1'-y_4)\in\overline{N}_{3,3}$ and $\mathcal{A}_0$ is an orbit S-ring, there exists an integer $d$ such that $d\cdot cL_0=(a^2-1)L_0$. Since $((d-a-1)c+a+2)L_0=\varphi_0\pi(y_1-y_4)\in\overline{N}_{3,3}$, there exists an integer $e$ such that $e\cdot cL_0=(a+2)L_0$. Since $((1-2e)c+3)L_0=\varphi_0\pi(y_2-y_5')\in\overline{N}_{3,3}$, there exists an integer $f$ such that $f\cdot cL_0=3L_0$. It follows that $fc\equiv3~({\rm mod}~|U_0/L|)$. Since $R_{3,3}$ is symmetric, from Lemma \ref{sym-1} \ref{sym-1-1}, $cL_0$ is not a generator of $U_0/L$. Then $3L_0$ is not a generator of $U_0/L$. Hence, $3\mid|U_0/L|$, and so $(fc,|U_0/L|)=3$. The fact $(c,|U_0/L|)\neq1$ then implies $(c,|U_0/L|)=3$. Let $h$ be the order of $cL_0$ in $U_0/L$.  Then $|U_0/L|=3h$. For each $qL_0\in\overline{N}_{3,3}$, since $\mathcal{A}_0$ is an orbit S-ring, the order of $qL_0$ in $U_0/L$ is also $h$, which implies $3\mid q$ and $(q/3,h)=1$. Thus, \ref{q=3 step7-2} and \ref{q=3 step7-4} are both valid.

Since $cL_0\in\overline{N}_{3,3}$, from Lemma \ref{sym}, there exists an integer $u\in\{2,3,\ldots,3h-1\}$ such that $ucL_0=L$ and $uL_0\in(\overline{N}_{1,2}+\overline{N}_{1,2})\setminus(\overline{N_{1,2}}\cup\overline{N}_{2,1})$. It follows that $3h\mid uc$. Since $3\mid c$ and $(c/3,h)=1$, one has $u=h$ or $2h$. By Lemma \ref{q=3 step4} \ref{q=3 step4-2}, we get $uL_0\in\overline{N}_{2,t}$ for some $t\in\{3,4\}$. Then $\overline{N}_{2,t}\subseteq\overline{N}_{1,2}+\overline{N}_{1,2}$, so
\begin{align}
R_{2,t}\in R_{1,2}^2\setminus\{R_{1,2},R_{2,1}\}.\label{2,t}
\end{align}
In view of Lemma \ref{nonsymmetric}, one gets $R_{t,2}\notin R_{1,2}^2$, which implies that $R_{2,t}$ is non-symmetric. Since $\mathcal{A}_0$ is an orbit S-ring, we have $\overline{N}_{2,t}=\{uL_0\}$. Lemma \ref{s=4} implies $t=4$. Thus, \ref{q=3 step7-5} is valid.

By \eqref{2,t}, we have $s=4$. Since $\varphi_0\pi(y_3)=\varphi_0\pi(y_2)+aL_0$, we have $\overline{N}_{3,3}\subseteq\overline{N}_{2,4}+\overline{N}_{1,2}$ by \eqref{a}. Since $|\overline{N}_{2,4}|=1$, one gets $R_{2,4}R_{1,2}=\{R_{3,3}\}$, and so $R_{4,2}\in R_{3,3}R_{1,2}$. Thus, \ref{q=3 step7-3} is also valid.
\end{proof}

\begin{lemma}\label{3 nmod f}
Suppose $R_{2,4}\in R_{1,2}^2\setminus\{R_{1,2},R_{2,1}\}$. If $R_{2,3}\notin R_{1,2}^2\setminus\{R_{1,2},R_{2,1},R_{2,4}\}$ or $R_{2,1}\in R_{2,3}^2$, then $\overline{N}_{1,2}\subseteq(1+3\mathbb{Z})L_0$, where $L_0\in\overline{N}_{1,2}$.
\end{lemma}
\begin{proof}
Let $L_0\in\overline{N}_{1,2}$. By Lemma \ref{q=3 step7} \ref{q=3 step7-5}, we may assume $\overline{N}_{2,4}=\{uL_0\}$ for some $u\in\{\pm h\}$. Since $R_{2,4}R_{1,2}=\{R_{3,3}\}$ from Lemma \ref{q=3 step7} \ref{q=3 step7-3}, one has $(u+1)L_0\in\overline{N}_{2,4}+\overline{N}_{1,2}=\overline{N}_{3,3}$. By Lemma \ref{q=3 step7} \ref{q=3 step7-4}, we obtain $3\mid u+1$, and so $3\nmid u$.

Pick $L_1\in\overline{N}_{1,2}$. Since each element in $\overline{N}_{1,2}$ is a generator of $U_0/L$ from Lemma \ref{genrators} \ref{genrators-3}, there exists an integer $f$ such that $L_1=fL_0$ with $3\nmid f$. Since $(f+1)L_0\in\overline{N}_{1,2}+\overline{N}_{1,2}$,  from Lemma \ref{q=3 step4} \ref{q=3 step4-2}, one has $(f+1)L_0\in\overline{N}_{1,2}\cup\overline{N}_{2,1}\cup\overline{N}_{2,3}\cup\overline{N}_{2,4}$. If $(f+1)L_0\in\overline{N}_{1,2}\cup\overline{N}_{2,1}\cup\overline{N}_{2,4}$, then $3\nmid f+1$ since $\overline{N}_{2,4}=\{uL_0\}$ with $3\nmid u$, which implies $f\equiv1~({\rm mod}~3)$. Now we suppose $(f+1)L_0\in\overline{N}_{2,3}\setminus(\overline{N}_{1,2}\cup\overline{N}_{2,1}\cup\overline{N}_{2,4})$. Then $R_{2,3}\in R_{1,2}^2\setminus\{R_{1,2},R_{2,1},R_{2,4}\}$, and so $R_{2,1}\in R_{2,3}^2$. Since each element in $\overline{N}_{1,2}$ is a generator of $U_0/L$ and $\mathcal{A}_0$ is an orbit S-ring, each element in $\overline{N}_{2,3}$ is also a generator of $U_0/L$. It follows that $3\nmid f+1$, and so $f\equiv1~({\rm mod}~3)$. Then $L_1\in\overline{N}_{1,2}\subseteq(1+3\mathbb{Z})L_0$. Since $L_1\in\overline{N}_{1,2}$ was arbitrary, the desired result follows.
\end{proof}

To distinguish the notations between the association schemes $(V\Gamma,\{\Gamma_{\tilde{i}}\}_{\tilde{i}\in\tilde{\partial}(\Gamma)})$ and $(U_0/L,\{R_{\tilde{i}}\}_{\Gamma_{\tilde{i}}\in F_{3}})$, let $\hat{k}_{\tilde{i}},\hat{p}_{\tilde{i},\tilde{j}}^{\tilde{j}}$ be the parameters of $(U_0/L,\{R_{\tilde{i}}\}_{\Gamma_{\tilde{i}}\in F_{3}})$.

\begin{lemma}\label{1,2^2}
Suppose that $R_{2,s}\in R_{1,2}^2\setminus\{R_{1,2},R_{2,1}\}$ and no element in $\overline{N}_{2,s}$ is a generator of $U_0/L$ for some $s\in\{3,4\}$. Then the following hold:
\begin{enumerate}
\item\label{1,2^2-1} $\hat{k}_{3,3}=\hat{k}_{1,2}$ ;

\item\label{1,2^2-2} $R_{1,2}\notin R_{1,2}^2$ and $R_{1,2}^2=\{R_{2,1},R_{2,4}\}$;

\item\label{1,2^2-3} $\hat{p}_{(3,3),(1,2)}^{(1,2)}=\hat{k}_{1,2}-1$;

\item\label{1,2^2-4} $R_{1,2}^3=\{R_{0,0},R_{3,3}\}$.
\end{enumerate}
\end{lemma}
\begin{proof}
Let $L_0\in\overline{N}_{1,2}$. By Lemma \ref{q=3 step7} \ref{q=3 step7-5}, we may assume $\overline{N}_{2,4}=\{uL_0\}$ for some $u\in\{\pm h\}$. By Lemma \ref{q=3 step7} \ref{q=3 step7-3}, we have
\begin{align}
R_{2,4}R_{1,2}=\{R_{3,3}\}.\label{2,41,2}
\end{align}
Since $\hat{k}_{2,4}=1$, from Lemma \ref{jb} \ref{jb-1}, \ref{1,2^2-1} is valid.

Since $R_{2,4}\in R_{1,2}^2\setminus\{R_{1,2},R_{2,1}\}$ from Lemma \ref{q=3 step7} \ref{q=3 step7-3} and $\hat{k}_{2,4}=1$, we have $\hat{p}_{(1,2),(1,2)}^{(2,4)}=\hat{k}_{1,2}$ by Lemma \ref{jb} \ref{jb-5}, which implies
\begin{align}
\overline{N}_{1,2}=uL_0-\overline{N}_{1,2},\label{uL_0}
\end{align}
and so $(L_0,uL_0)\in R_{1,2}$.

Suppose $R_{1,2}\in R_{1,2}^2$. Then $R_{1,2}\in R_{1,2}R_{2,1}$. It follows that there exists $L'\in U_0/L$ such that $(uL_0,L'),(L_0,L')\in R_{1,2}$. Then $L'=L'-L_0+L_0\in\overline{N}_{1,2}+\overline{N}_{1,2}$. Since $L'=uL_0+L'-uL_0\in\overline{N}_{2,4}+\overline{N}_{1,2}\subseteq\overline{N}_{3,3}$ from \eqref{2,41,2}, contrary to Lemma \ref{nonsymmetric}. Then $R_{1,2}\notin R_{1,2}^2$. By Lemma \ref{q=3 step4} \ref{q=3 step4-2}, we obtain $R_{1,2}^2=\{R_{2,1},R_{2,4}\}$ or $\{R_{2,1},R_{2,3},R_{2,4}\}$.

To complete the proof of \ref{1,2^2-2}, assume the contrary, namely, $R_{1,2}^2\neq\{R_{2,1},R_{2,4}\}$. Then $R_{2,3}\in R_{1,2}^2$. By Lemma \ref{claim} \ref{claim-1}, we have $\Gamma_{2,3}\in\Gamma_{1,2}^2$. In view of Lemma \ref{q=3 step6}, one gets $\Gamma_{3,2}\in\Gamma_{1,2}^3$, and so $R_{3,2}\in R_{1,2}^3$. Then there exists a circuit $(L,L_0,L_1,L_2,L_3)$ in ${\rm Cay}(U_0/L,\overline{N}_{1,2})$ such that $L_1\in\overline{N}_{2,3}$. Since $L_i-L_{i+2}\in\overline{N}_{1,2}+\overline{N}_{1,2}+\overline{N}_{1,2}$ for $0\leq i\leq 3$ with $L_4=L$ and $L_5=L_0$, from Lemma \ref{nonsymmetric}, we get $L_{i+2}-L_i\notin\overline{N}_{2,4}$. Since $R_{1,2}^2=\{R_{2,1},R_{2,3},R_{2,4}\}$, we obtain $L_{i+2}-L_i\in \overline{N}_{2,1}\cup\overline{N}_{2,3}$.

We claim $R_{2,1}\in R_{2,3}^2$. Since $L_3-L_1,-L_2\in\overline{N}_{2,1}\cup\overline{N}_{2,3}$, we first assume $L_3-L_1$ or $-L_2\in\overline{N}_{2,3}$. Since $L_1\in\overline{N}_{2,3}$, one has $L_3-L\in\overline{N}_{2,3}+\overline{N}_{2,3}$ or $L_1-L_2\in\overline{N}_{2,3}+\overline{N}_{2,3}$, which implies $R_{2,1}\in R_{2,3}^2$. Now we suppose $L_3-L_1,-L_2\in\overline{N}_{2,1}$. The fact $R_{1,2}\notin R_{1,2}^2$ implies $R_{2,1}\notin R_{1,2}R_{2,1}$. Since $L_0-L_3=L_0-L_1+L_1-L_3\in\overline{N}_{2,1}+\overline{N}_{1,2}$ and $-(-L_2)+(-L_0)\in\overline{N}_{1,2}+\overline{N}_{2,1}$, we have $L_0-L_3,L_2-L_0\in\overline{N}_{2,3}$,  which implies $R_{2,1}\in R_{2,3}^2$. Thus, our claim is valid.

Since $L_0$ is a generator of $U_0/L$ from Lemma \ref{genrators} \ref{genrators-3}, we may set $L_i=f_iL_0$ with $f_i\in\mathbb{Z}$ for $0\leq i\leq3$. By the claim and Lemma \ref{3 nmod f}, we have $-f_3\equiv f_i-f_{i-1}\equiv1~({\rm mod}~3)$ for $1\leq i\leq 3$, contrary to the fact $f_3\equiv(f_3-f_2)+(f_2-f_1)+(f_1-1)+1\equiv 1~({\rm mod}~3)$. Thus, \ref{1,2^2-2} is valid.

Since $\hat{k}_{2,4}=1$ and $\hat{p}_{(1,2),(1,2)}^{(2,4)}=\hat{k}_{1,2}$, from Lemma \ref{jb} \ref{jb-1}, we have $\hat{p}_{(2,1),(2,1)}^{(1,2)}=\hat{p}_{(1,2),(1,2)}^{(2,1)}=\hat{k}_{1,2}-1$.  Since $\overline{N}_{2,4}=\{uL_0\}$, one has $L_0+(\overline{N}_{1,2}\setminus\{(u-1)L_0\})\subseteq\overline{N}_{2,1}$. By \eqref{uL_0}, one gets
\begin{align}
(u+1)L_0+(\overline{N}_{1,2}\setminus\{(u-1)L_0\})\subseteq uL_0+\overline{N}_{2,1}=\overline{N}_{1,2}\nonumber.
\end{align}
By \eqref{2,41,2}, one gets $(u+1)L_0\in\overline{N}_{2,4}+\overline{N}_{1,2}=\overline{N}_{3,3}$, which implies $\hat{p}_{(2,1),(1,2)}^{(3,3)}=\hat{p}_{(1,2),(2,1)}^{(3,3)}\geq\hat{k}_{1,2}-1$. By \ref{1,2^2-1}, we have $\hat{k}_{3,3}=\hat{k}_{1,2}$, and so $k_{1,2}^2\geq k_{1,2}+\hat{p}_{(1,2),(2,1)}^{(3,3)}k_{3,3}\geq k_{1,2}^2$. Thus, we obtain $\hat{p}_{(1,2),(2,1)}^{(3,3)}=\hat{k}_{1,2}-1$ and $R_{1,2}R_{2,1}=\{R_{0,0},R_{3,3}\}$ from Lemma \ref{jb} \ref{jb-1}. Then \ref{1,2^2-1}, Lemma \ref{jb} \ref{jb-2} and the former imply \ref{1,2^2-3}, while \eqref{2,41,2}, \ref{1,2^2-2} and the latter imply \ref{1,2^2-4}.
\end{proof}

\subsection{Proof of Proposition \ref{q=3}}

Let $l=|L|$ and $n_i=|U_i/L|$ for $0\leq i\leq k$. By Lemma \ref{genrators} \ref{genrators-2},
\begin{align}
\Delta_3\cong({\rm Cay}(U_0/L,\overline{N}_{1,2})\times K_{n_1}\times\cdots\times K_{n_k})[\overline{K}_{l}],\label{delta3}
\end{align}
where $n_i$, $i\in\{0,1,\ldots,k\}$, are pairwise coprime, and $n_i>3$ for all $i>0$.

\begin{lemma}\label{digraph-1}
The following hold:
\begin{enumerate}
\item\label{digraph-1-1} if ${\rm Cay}(U_0/L,\overline{N}_{1,2})\cong C_3$, then $k=0$ or $1$;

\item\label{digraph-1-2} if ${\rm Cay}(U_0/L,\overline{N}_{1,2})\cong P(p)$ for some prime $p>3$ with $p\equiv3\pmod{4}$, then $k=0$;

\item\label{digraph-1-3} if ${\rm Cay}(U_0/L,\overline{N}_{1,2})\cong{\rm Cay}(\mathbb{Z}_{13},\{1,3,9\})$, then $k=0$.
\end{enumerate}
\end{lemma}
\begin{proof}
Let $\Delta=C_3,P(p),{\rm Cay}(\mathbb{Z}_{13},\{1,3,9\})$, in \ref{digraph-1-1} \ref{digraph-1-2} and \ref{digraph-1-3}, respectively. By \eqref{delta3}, there exists an isomorphism $\sigma$ from $(\Delta\times K_{n_1}\times\cdots\times K_{n_k})[\overline{K}_{l}]$ to $\Delta_{3}$, where $n_i>3$ for all $i>0$. We may assume $\sigma(0,0,\ldots,0)=0$.

\ref{digraph-1-1} Let $a=\sigma(1,1,\ldots,1,0)$ and $b=\sigma(2,0,0,\ldots,0)$. Suppose $k>1$. Let $c=\sigma(2,2,0,0,\ldots,0)$. It follows that $b,c\notin N_{1,2}\cup N_{2,1}$ and $\partial_{\Delta_3}(b,0)\equiv\partial_{\Delta_3}(c,0)\equiv 1$~(mod $3$). Then $\partial_{\Delta_3}(b,0),\partial_{\Delta_3}(c,0)\geq4$. Since $(0,a,b)$ and $(0,a,c)$ are paths in $\Delta_3$, we have $b,c\in N_{1,2}+N_{1,2}$. By Lemma \ref{q=3 step4} \ref{q=3 step4-1}, we have $b,c\in N_{2,3}\cup N_{2,4}$. If $b$ or $c\in N_{2,3}$, then $\Gamma_{2,3}\in\Gamma_{1,2}^2$, and so $\Gamma_{3,2}\in\Gamma_{1,2}^3$ from Lemma \ref{q=3 step6}, which implies $\partial_{\Delta_3}(b,0)\leq3$ or $\partial_{\Delta_3}(c,0)\leq3$, a contradiction. Then $b,c\in N_{2,4}$. But the sets
\begin{align}
P_{(1,2),(1,2)}(0,b)&=\sigma(\{(1,x_1,\ldots,x_k,y)\mid y\in\mathbb{Z}_{l}, x_i\neq0~{\rm for}~1\leq i\leq k\}),\nonumber\\
P_{(1,2),(1,2)}(0,c)&=\sigma(\{(1,x_1,\ldots,x_k,y)\mid y\in\mathbb{Z}_{l}, x_1\neq 2, x_{i}\neq0~{\rm for}~1\leq i\leq k\})\nonumber
\end{align}
have different size. This is impossible.

\ref{digraph-1-2} Let $a=\sigma(1,0,0,\ldots,0)$. Since $p>3$, there exist $i,j\in {\rm GF}(p)$ such that $i,j,1-i$ and $-1-j$ are squares in the multiplicative group of GF$(p)$. Let $b=\sigma(i,1,\ldots,1,0)$ and $c=\sigma(1+j,1,\ldots,1,0)$.

Suppose $k>0$. Then $a\notin N_{1,2}\cup N_{2,1}$. Since $(0,b,a,c)$ is a circuit in $\Delta_3$, one gets $a\in(N_{1,2}+N_{1,2})\cap(N_{2,1}+N_{2,1})$. By Lemma \ref{q=3 step4} \ref{q=3 step4-1}, one gets $a\in N_{1,2}\cup N_{2,1}$, a contradiction.

\ref{digraph-1-3} Let $a=\sigma(3,1,\ldots,1,0)$, $b=\sigma(4,0,0,\ldots,0)$ and $c=\sigma(6,0,0,\ldots,0)$. Suppose $k>0$. Since $(0,a,b)$ and $(0,a,c)$ are paths in $\Delta_{3}$, by Lemma \ref{q=3 step4} \ref{q=3 step4-1}, we get $b,c\in N_{2,3}\cup N_{2,4}$. Since $(b,b+a,b+2a,0)$ and $(c,c+a,c+2a,0)$ are paths in $\Delta_{3}$, one has $b,c\in N_{2,3}$. But the sets
\begin{align}
P_{(1,2),(1,2)}(0,b)&=\sigma(\{(x,x_1,\ldots,x_k,y)\mid y\in\mathbb{Z}_{l}, x\in \{1,3\}, x_i\neq0~{\rm for}~1\leq i\leq k\}),\nonumber\\
P_{(1,2),(1,2)}(0,c)&=\sigma(\{(3,x_1,\ldots,x_k,y)\mid y\in\mathbb{Z}_{l}, x_{i}\neq0~{\rm for}~1\leq i\leq k\})\nonumber
\end{align}
have different size. This is impossible.
\end{proof}

Now we are ready to complete the proof of Proposition \ref{q=3}.

\begin{proof}[Proof of Proposition~\ref{q=3}] We divide our proof into two cases.

\textbf{Case 1.} $R_{3,3}\in R_{1,2}R_{2,1}\setminus\{R_{0,0}\}$.

Since $R_{3,3}$ is symmetric, from Lemma \ref{sym} and Lemma \ref{q=3 step4} \ref{q=3 step4-2}, there exists $s\in\{3,4\}$ such that $R_{2,s}\in R_{1,2}^2\setminus\{R_{1,2},R_{2,1}\}$ and no element in $\overline{N}_{2,s}$ is a generator of $U_0/L$. By Lemma \ref{1,2^2} \ref{1,2^2-4}, we have $R_{1,2}^3=\{R_{0,0},R_{3,3}\}$.

By Lemma \ref{1,2^2} \ref{1,2^2-3}, we have $\hat{p}_{(3,3),(1,2)}^{(1,2)}=\hat{k}_{1,2}-1$. By Lemma \ref{q=3 step7} \ref{q=3 step7-3}, we have $R_{4,2}\in R_{3,3}R_{1,2}$. Since $\hat{k}_{3,3}=\hat{k}_{1,2}$ from Lemma \ref{1,2^2} \ref{1,2^2-1}, we get $\hat{p}_{(3,3),(1,2)}^{(4,2)}\hat{k}_{4,2}\geq \hat{k}_{1,2}$ by Lemma \ref{jb} \ref{jb-5}.  This forces $\hat{k}_{1,2}^2=\hat{k}_{3,3}\hat{k}_{1,2}=\hat{p}_{(3,3),(1,2)}^{(1,2)}\hat{k}_{1,2}+\hat{p}_{(3,3),(1,2)}^{(4,2)}\hat{k}_{4,2}$. It follows from Lemma \ref{jb} \ref{jb-1} that $R_{3,3}R_{1,2}=\{R_{1,2},R_{4,2}\}$, and so $R_{1,2}^4=\{R_{1,2},R_{4,2}\}$.

By Lemma \ref{q=3 step7} \ref{q=3 step7-5}, we have $\hat{k}_{2,4}=1$. Since $R_{2,4}\in R_{1,2}^2$ from Lemma \ref{1,2^2} \ref{1,2^2-2}, we have $R_{2,1}\in R_{4,2}R_{1,2}$. Lemma \ref{jb} \ref{jb-1} implies $R_{4,2}R_{1,2}=\{R_{2,1}\}$, and so $R_{1,2}^5=\{R_{2,1},R_{2,4}\}$. Thus,
\begin{align}
\mathcal{S}(\mathcal{A}_0)=\{\{L\},\overline{N}_{1,2},\overline{N}_{2,1},\overline{N}_{2,4},\overline{N}_{3,3},\overline{N}_{4,2}\}.\label{A_0}
\end{align}

Let $L_0\in\overline{N}_{1,2}$. By Lemma \ref{q=3 step7} \ref{q=3 step7-2}, we have $|U_0/L|=3h$ with $h\in\mathbb{Z}$. Lemma \ref{q=3 step7} \ref{q=3 step7-5} implies $\overline{N}_{2,4}\cup\overline{N}_{4,2}=\{\pm hL_0\}=h\mathbb{Z}L_0\setminus\{L\}$. By Lemma \ref{genrators} \ref{genrators-3},  each element in $\overline{N}_{1,2}$ is a generator of $U_0/L$. Lemma \ref{q=3 step7} \ref{q=3 step7-4} implies that the order of each element in $\overline{N}_{3,3}$ is $h$.
Since $\mathcal{A}_0$ is an orbit S-ring, from \eqref{A_0}, the order of each element of $U_0/L$ is $1,3,h$ or $3h$. This implies that $h$ is a prime. Then $\overline{N}_{3,3}=3\mathbb{Z}L_0\setminus\{L\}$, which forces $h>3$. By \eqref{A_0}, $\overline{N}_{1,2}\cup\overline{N}_{2,1}$ consists of all generators of $U_0/L$. Lemma \ref{3 nmod f} implies $\overline{N}_{1,2}=((1+3\mathbb{Z})\setminus h\mathbb{Z})L_0$. Since $3\nmid h$, we have $U_0/L=\langle hL_0\rangle\times\langle 3L_0\rangle$ and
\begin{align}
\overline{N}_{1,2}=((\pm h+3\mathbb{Z})\setminus h\mathbb{Z})L_0=\pm\{hL_0\}+3(\mathbb{Z}\setminus h\mathbb{Z})L_0.\nonumber
\end{align}
It follows that ${\rm Cay}(U_0/L,\overline{N}_{1,2})\cong C_3\times K_{h}$. By Lemma \ref{digraph-1} \ref{digraph-1-1} and \eqref{delta3}, $\Delta_{3}$ is isomorphic to $(C_3\times K_{h})[\overline{K}_{l}]$.

\textbf{Case 2.} $R_{3,3}\notin R_{1,2}R_{2,1}\setminus\{R_{0,0}\}$.

By Lemma \ref{q=3 step7} \ref{q=3 step7-1}, we have $R_{2,4}\notin R_{1,2}^2\setminus\{R_{1,2},R_{2,1}\}$, which implies $R_{1,2}^{2}\subseteq\{R_{1,2},R_{2,1},R_{2,3}\}$ from Lemma \ref{q=3 step4} \ref{q=3 step4-2}.

\textbf{Case 2.1} $\Gamma_{2,3}\notin F_3$ or $R_{2,3}\in\{R_{0,0},R_{1,2},R_{2,1}\}$.

By Lemma \ref{genrators} \ref{genrators-5}, we have $R_{0,0}\notin R_{1,2}^2$, and so $R_{1,2}^{2}\subseteq\{R_{1,2},R_{2,1}\}$. By Lemma \ref{genrators} \ref{genrators-3}, each element in $\overline{N}_{1,2}+\overline{N}_{1,2}\subseteq\overline{N}_{1,2}\cup\overline{N}_{2,1}$ is a generator of $U_0/L$. Lemma \ref{sym-1} \ref{sym-1-2} implies that $(U_0/L,\{R_{\tilde{i}}\}_{\Gamma_{\tilde{i}}\in F_3})$ is skew-symmetric. By Lemma \ref{q=3 step5}, one obtains $R_{1,2}R_{2,1}\subseteq\{R_{0,0},R_{1,2},R_{2,1}\}$. It follows that $(U_0/L,\{R_{0,0},R_{1,2},R_{2,1}\})$ is a 2-class non-symmetric association scheme. Since each element in $\overline{N}_{1,2}+\overline{N}_{1,2}\supseteq\overline{N}_{2,1}$ is a generator of $U_0/L$, $p:=|U_0/L|$ is a prime. By Lemma \ref{prime}, $(U_{0}/L,\{R_{0,0},R_{1,2},R_{2,1}\})$ is the attached scheme of Paley digraph $P(p)$ with $p\equiv3$ (mod~$4$). It follows that ${\rm Cay}(U_0/L,\overline{N}_{1,2})\cong P(p)$. By \eqref{delta3}, we obtain $\Delta_{3}\cong(P(p)\times K_{n_1}\times\cdots\times K_{n_k})[\overline{K}_{l}]$ with $p\nmid n_i$ and $n_i>3$ for all $i>0$.

If $p=3$, from Lemma \ref{digraph-1} \ref{digraph-1-1}, then $\Delta_{3}$ is isomorphic to $C_3[\overline{K}_l]$ or $(C_3\times K_{n_1})[\overline{K}_l]$ with $3\nmid n_1$ and $n_1>3$. If $p>3$, from Lemma \ref{digraph-1} \ref{digraph-1-2}, then $\Delta_{3}$ is isomorphic to one of the digraphs in Theorem \ref{Main2} \ref{Main2-2}.

\textbf{Case 2.2.} $\Gamma_{2,3}\in F_3$ and $R_{2,3}\notin\{R_{0,0},R_{1,2},R_{2,1}\}$.

If $R_{2,3}\in R_{1,2}^2$, then $R_{2,3}$ is non-symmetric since $R_{3,2}\notin R_{1,2}^2$ from Lemma \ref{nonsymmetric}; if $R_{2,3}\notin R_{1,2}^2$, then $R_{1,2}^2\subseteq\{R_{1,2},R_{2,1}\}$, and so $R_{2,3}$ is non-symmetric from Lemma \ref{sym-1} \ref{sym-1-2}.

Since $\mathcal{A}_0$ is an orbit S-ring, from Lemma \ref{q=3 step7} \ref{q=3 step7-1},  $R_{2,3}\notin R_{1,2}^2$ or each element in $\overline{N}_{2,3}$ is a generator of $U_0/L$. Since $R_{1,2}^{2}\subseteq\{R_{1,2},R_{2,1},R_{2,3}\}$, from Lemma \ref{genrators} \ref{genrators-3}, each element in $\overline{N}_{1,2}+\overline{N}_{1,2}\subseteq\overline{N}_{1,2}\cup\overline{N}_{2,1}\cup\overline{N}_{2,3}$ is a generator of $U_0/L$. Lemma  \ref{sym-1} \ref{sym-1-2} implies that $(U_0/L,\{R_{\tilde{i}}\}_{\Gamma_{\tilde{i}}\in F_3})$ is skew-symmetric. By Lemma \ref{q=3 step5}, one obtains $R_{1,2}R_{2,1}\subseteq\{R_{0,0},R_{1,2},R_{2,1},R_{2,3},R_{3,2}\}$. Lemma \ref{jb4} implies $\hat{k}_{1,2}\geq\hat{k}_{2,3}$. It follows from Theorem \ref{main1} that $\mathfrak{X}:=(U_0/L,\{R_{0,0},R_{1,2},R_{2,1},R_{2,3},R_{3,2}\})$ is a 4-class skew-symmetric association scheme.

Suppose that $|U_0/L|$ is not a prime. Since $\mathcal{A}_0$ is an orbit S-ring with rank $5$, from Lemma \ref{genrators} \ref{genrators-3}, $\overline{N}_{1,2}\cup\overline{N}_{2,1}$ consists of all generators of $U_0/L$. Then the order of each element of $\overline{N}_{2,3}\cup\overline{N}_{3,2}$ is the unique nontrivial divisor of $|U_0/L|$. This implies that $|U_0/L|=p^2$ for some odd prime $p$. Then $\hat{k}_{1,2}=p(p-1)/2$. Since $\mathcal{A}_0$ is an orbit S-ring, there exists $K\leq {\rm Aut}(U_0/L)$ such that $\mathcal{A}_0=\mathcal{O}(K,U_0/L)$ and $p\mid |K|$. Lemma \ref{p^2} implies that $\mathcal{A}_0$ is not free, a contradiction. Thus, $|U_0/L|$ is a prime.

By Lemma \ref{prime}, $\mathfrak{X}$ is a cyclotomic and pseudocyclic scheme. In view of Theorem \ref{main1}, one has ${\rm Cay}(U_0/L,\overline{N}_{1,2})\cong{\rm Cay}(\mathbb{Z}_{13},\{1,3,9\})$. By Lemma \ref{digraph-1} \ref{digraph-1-3} and \eqref{delta3}, $\Delta_{3}$ is isomorphic to one of the digraphs in Theorem \ref{Main2} \ref{Main2-4}.
\end{proof}

\section{Proofs of Theorem \ref{Main2} and Corollary \ref{Main3}}

\begin{proof}[Proof of Theorem~\ref{Main2}]
Note that $C_m$, $P(p)$ and ${\rm Cay}(\mathbb{Z}_{13},\{1,3,9\})$ are all weakly distance-regular. By \cite[Proposition 2.4]{KSW03}, all digraphs in Theorem \ref{Main2} \ref{Main2-1}, \ref{Main2-2} and \ref{Main2-4} are also weakly distance-regular. One can verify that all digraphs in Theorem \ref{Main2} \ref{Main2-3} are weakly distance-regular. It is easy to see that all digraphs in Theorem \ref{Main2} are circulants. This completes the proof of the sufficiency part.

We now prove the necessity. Since $\Gamma$ is a weakly distance-regular circulant of one type of arcs, we have $T=\{g\}$, where $g>2$ is the girth of $\Gamma$. Since $\Gamma$ is connected, $\Gamma_{1,g-1}$ generates the attached scheme of $\Gamma$. It follows that $F_{g}=\{\Gamma_{\tilde{i}}\}_{\tilde{i}\in\tilde{\partial}(\Gamma)}$, and so $F_{g}(x)=V\Gamma$ for any $x\in V\Gamma$. Hence, $\Gamma=\Delta_{g}$. If $g>3$, then from Proposition \ref{(1,q-1)} $\Gamma$ is isomorphic to one of the digraph in \ref{Main2-1} since $p_{(1,g-1),(1,g-1)}^{(2,g-2)}>0$. Now suppose $g=3$. By Proposition \ref{q=3}, it suffices to show that $(C_3\times K_{h})[\overline{K}_{l}]$ is not weakly distance-regular, where $h>3$, $3\nmid h$ and $l>1$.

Suppose that $\Gamma:=(C_3\times K_{h})[\overline{K}_{l}]$ is a weakly distance-regular digraph. Observe that $(0,1,0),(0,0,1)\in\Gamma_{3,3}(0,0,0)$. But the sets
\begin{align}
P_{(1,2),(2,1)}((0,0,0),(0,1,0))&=\{(1,x,y)\mid x\notin \{0,1\}, y\in\mathbb{Z}_{l}\},\nonumber\\
P_{(1,2),(2,1)}((0,0,0),(0,0,1))&=\{(1,x,y)\mid x\neq0, y\in\mathbb{Z}_{l}\}\nonumber
\end{align}
have different size. This is impossible.
\end{proof}

\begin{remark}
In fact, there are weakly distance-regular circulants $\Gamma$ consisting of arcs of more than one types such that $\Delta_3$ are isomorphic to digraphs in Proposition \ref{q=3} \ref{q=3-3}. For example, let $\Gamma=(C_3\times K_h)[C_4]$ with $h>3$ and $3\nmid h$. Then $\Gamma$ is a weakly distance-regular circulant consisting of arcs of types $(1,2)$ and $(1,3)$ such that $\Delta_3=(C_3\times K_h)[\overline{K}_4]$.
\end{remark}

\begin{proof}[Proof of Corollary~\ref{Main3}]
Observe that all the digraphs in Corollary~\ref{Main3} are
primitive weakly distance-regular circulants.

Let $\Gamma=\textrm{Cay}(\mathbb{Z}_{p},S)$ be a primitive weakly distance-regular digraph of girth $g$.  Observe that $\mathfrak{X}=(V\Gamma,\{\Gamma_{\tilde{i}}\}_{\tilde{i}\in\tilde{\partial}(\Gamma)})$ is a translation scheme. By Lemma~\ref{prime}, $p$ is a prime and $\mathfrak{X}$
is a cyclotomic scheme. 

Suppose $(1,q)\in\tilde{\partial}(\Gamma)$ and pick $x_1\in N_{1,q}$. By Lemma \ref{lem:1}, there exists an integer $s$ such that $(N_{1,g-1})^{(s)}=N_{1,q}$, which implies that
there exists $y_1\in N_{1,g-1}$ such that $sy_1=x_1$. Then there exists a circuit
$(y_1,y_2,\dots,y_g=0)$ of length $g$ consisting of arcs of type $(1,g-1)$. Multiplication
by $s$ gives a path of length $g-1$ from $x_1$ to $0$, which implies $x_1\in N_{1,g-1}$.
This implies $q=g-1$. Thus, $\Gamma$ only contains one type of arcs. The desired result follows from Theorem \ref{Main2}.
\end{proof}

\section*{Acknowledgements}
A.~Munemasa is supported by JSPS Kakenhi (Grant No. 17K05155 and No. 20K03527),
K. Wang is supported by the National Key R$\&$D Program of China (No.~2020YFA0712900) and NSFC (12071039, 12131011), Y.~Yang is supported by NSFC (12101575) and the Fundamental Research Funds for the Central Universities (2652019319).
A part of this research was done while A.~Munemasa was visiting Beijing Normal University,
and while Y.~Yang was visiting Tohoku University. We acknowledge the support
from RACMaS, Tohoku University.

\end{document}